\documentclass[11pt]{amsart}
\usepackage{amsfonts,amssymb,amsthm,amsmath,amsxtra,amscd,verbatim,eucal}
\usepackage[all]{xy}
\usepackage[dvips]{graphics}
\usepackage{graphicx}

\theoremstyle{plain}
\newtheorem{thm}{Theorem}[section]

\newtheorem*{thmA}{Theorem~A}
\newtheorem*{thmB}{Theorem~B}
\newtheorem*{thmC}{Theorem~C}

\newtheorem{lem}[thm]{Lemma}
\newtheorem{prop}[thm]{Proposition}
\newtheorem{cor}[thm]{Corollary}

\theoremstyle{theorem}
\newtheorem{defi}[thm]{Definition}

\theoremstyle{remark}
\newtheorem{eg}[thm]{Example}

\newtheorem{rmk}[thm]{Remark}

\def\Z{{\mathbb Z}}
\def\N{{\mathbb N}}

\def\C{{\mathbb C}}

\def\R{{\mathbb R}}
\def\Q{{\mathbb Q}}
\def\P{{\mathbb P}}

\def\O{\mathcal{O}}

\def\M{\mathcal{M}}

\def\J{\mathcal{J}}

\def\fra{\mathfrak{a}}
\def\frb{\mathfrak{b}}

\def\frd{\mathfrak{d}}

\def\frm{\mathfrak{m}}

\def\X{\mathfrak{X}}
\def\Y{\mathfrak{Y}}

\def\a{\alpha}
\def\b{\beta}

\def\d{\delta}
\def\f{\phi}

\def\ep{\varepsilon}

\def\l{\lambda}
\def\n{\nu}
\def\m{\mu}

\def\p{\pi}

\def\D{\Delta}

\def\Om{\Omega}

\def\.{\cdot}
\def\~{\widetilde}
\def\^{\widehat}

\def\ov{\overline}
\def\rat{\dashrightarrow}

\def\inj{\hookrightarrow}

\def\lrd{\lfloor}
\def\rrd{\rfloor}
\def\lru{\lceil}
\def\rru{\rceil}

\renewcommand{\and}{ \quad \text{and} \quad }

\DeclareMathOperator{\Spec} {Spec}

\DeclareMathOperator{\Pic} {Pic}

\DeclareMathOperator{\ord} {ord}

\DeclareMathOperator{\Div} {Div}
\DeclareMathOperator{\CDiv} {CDiv}

\DeclareMathOperator{\dv} {div}
\DeclareMathOperator{\Nef} {Nef}

\DeclareMathOperator{\Vol} {Vol}
\DeclareMathOperator{\vol} {vol}

\DeclareMathOperator{\Cl} {Cl}

\DeclareMathOperator{\Fitt} {Fitt}

\DeclareMathOperator{\Jac} {Jac}

\DeclareMathOperator{\Mov} {Mov}

\DeclareMathOperator{\Env} {Env}
\DeclareMathOperator{\Exc} {Exc}

\DeclareMathOperator{\loc} {loc}

\title{The volume of an isolated singularity}

\date{\today}
\author[S. Boucksom]{Sebastien Boucksom}
\address{CNRS - Institut de Math\'ematiques de Jussieu, 4 place Jussieu, 75252 Paris Cedex, France}
\email{boucksom@math.jussieu.fr}

\author[T. de Fernex]{Tommaso de Fernex}
\address{Department of Mathematics, University of Utah, 155 South 1400 East,
Salt Lake City, UT 48112-0090, USA}
\email{defernex@math.utah.edu}

\author[C. Favre]{Charles Favre}
\address{CNRS - Centre de Math\'ematiques Laurent Schwartz, 
\'Ecole Polytechnique, 
91128 Palaiseau Cedex, France}
\email{favre@math.polytechnique.fr}

\thanks{The second author is partially supported by NSF CAREER Grant DMS-0847059.
The last author is supported by the ANR-project BERKO}

\begin{document}

\begin{abstract}
We introduce a notion of volume of a normal isolated singularity 
that  generalizes Wahl's characteristic number of surface singularities to arbitrary dimensions. We prove a basic monotonicity property of this volume under finite morphisms.
We draw several consequences regarding the existence of non-invertible finite endomorphisms fixing an isolated singularity. Using a cone construction, we deduce that the anticanonical divisor of any smooth projective variety carrying a non-invertible polarized endomorphism is pseudoeffective. 

Our techniques build on Shokurov's $b$-divisors. We define
the notions of nef Weil $b$-divisors, and of nef envelopes of $b$-divisors. We relate the latter 
to the pull-back of Weil divisors introduced by de Fernex and Hacon.
Using the subadditivity theorem for multiplier ideals with respect to pairs
recently obtained by Takagi, we carry over to the isolated singularity case
the intersection theory of nef Weil $b$-divisors formerly developed by Boucksom, Favre, and Jonsson
in the smooth case.
\end{abstract}

\maketitle

\tableofcontents

%
%%%%%%%%%%%%%%%%%%%%%%
%

\section*{Introduction}

Wahl's characteristic number \cite{Wahl} is a topological invariant of the link of a 
normal surface singularity. Its simple behavior 
under finite morphisms enables one to characterize surface singularities that
carry finite non-invertible endomorphisms. Our main goal is to generalize Wahl's invariant  to higher dimensional isolated normal singularities, and to present a few applications to the description of singularities admitting non-trivial finite endomorphisms. 
Our main result can be stated as follows.

\begin{thmA} To any  normal isolated singularity $(X,0)$ there
is associated a non-negative real number  
$\Vol(X,0)$ that we call its \emph{volume}, satisfying the following properties:
\begin{enumerate}
\item[(i)] For every finite morphism $\phi\colon(X,0)\to(Y,0)$ of degree $e(\phi)$ we have 
$$
\Vol(X,0)\ge e(\phi)\Vol(Y,0),
$$ 
and equality holds when $\phi$ is \'etale in codimension one.
\item[(ii)] If $\dim X=2$ then $\Vol(X,0)$ coincides with Wahl's characteristic number.
\item[(iii)]  If $X$ is $\Q$-Gorenstein then 
$\Vol(X,0)=0$ if and only if $X$ has log-canonical (=lc) singularities.  
\end{enumerate}
\end{thmA}
Our result generalizes in particular the well-known fact that 
$\Q$-Gorenstein lc singularities are preserved under finite morphisms (see for instance~\cite[Proposition 3.16]{Kol}). 

Just as in dimension $2$, one infers restrictions on isolated singularities admitting finite endomorphisms.
\begin{thmB} 
Suppose $\f \colon (X,0) \to (X,0)$ is a finite non-invertible endomorphism  of an isolated singularity.
Then $\Vol(X,0) =0$. 

If $X$ is $\Q$-Gorenstein then $X$ has lc singularities, and it furthermore has klt singularities if $\f$ is not \'etale in codimension one.
\end{thmB}
To obtain a more precise classification of singularities carrying finite endomorphisms one would need to get deeper into the structure of singularities with $\Vol(X,0)=0$. This can be done in dimension $2$, see~\cite{Wahl,Fav}, but unfortunately, this task seems 
very difficult at the moment in arbitrary dimension.  To illustrate the previous result, we construct however several classes of (non-necessarily $\Q$-Gorenstein)  isolated normal singularities carrying finite endomorphisms, see \S\ref{sec:simple}-\ref{sec:cusp} below. 
Our examples include quotient singularities, Tsuchihashi's cusp 
singularities~\cite{oda,tsu}, toric singularities, and certain simple singularities obtained from cone or deformation constructions. 

\smallskip

In dimension $2$, the conclusion of Theorem~B plays a key role in the classification of projective surfaces admitting non-invertible endomorphisms, which is by now essentially complete, see~\cite{fuji-naka,naka}. In higher dimensions, classifying projective varieties carrying a non-invertible endomorphism has recently attracted quite a lot of attention, see~\cite{dqz} and the references therein, but the general problem remains largely open.

The assumption on the singularity being isolated in Theorem B is too strong to be directly useful in this perspective.  Nevertheless we observe that Theorem~B has some consequences in the more rigid case of so-called polarized endomorphisms. Recall that an endomorphism $\f \colon V \to V$ of a projective variety is said to be polarized if there exists an ample line bundle $L$ on $V$ such that $\phi^*L=dL$ in $\Pic(V)$ for some $d\ge 1$ (cf.~\cite{zhang} for a nice survey). By looking at the affine cone over $X$ induced by a large enough multiple of $L$, we obtain:

\begin{thmC} 
If $V$ is a smooth projective variety carrying a non-invertible polarized endomorphism $\f$ then $-K_V$ is pseudoeffective. 
\end{thmC}
Observe that the ramification formula implies $K_V \cdot L^{n-1} \le 0$. If $K_V$ is pseudoeffective then $K_V\equiv 0$ and $(V,\f)$ is then an endomorphism of an abelian variety up to finite \'etale cover (see \cite[Theorem 4.2]{Fakh}). If $K_V$ is not pseudoeffective then  $V$ is uniruled by \cite{BDPP}, and our result then puts further constraints on the geometry of $V$. 

Throughout the paper, we insist on working with arbitrary non $\Q$-Gorenstein singularities. 
This degree of generality is crucial to obtain Theorem~C since
the cone over $V$ is $\Q$-Gorenstein iff $\pm K_V$ is either $\Q$-linearly trivial or ample, see Example~\ref{eg:cone} below. 
$$
***
$$
In order to understand our construction, and the difficulties that one has to overcome
to define the volume above, let us recall briefly Wahl's definition for a normal surface singularity
$(X,0)$. 

Pick any \emph{log-resolution} $\p \colon Y \to X$ of $(X,0)$, i.e.~a birational morphism which is an isomorphism above $X\setminus\{0\}$, and such that $Y$ is smooth and the scheme-theoretic inverse image $\p^{-1}(0)$ is a divisor with simple normal crossing support $E$. Let $K_X$ be a canonical divisor on $X$ and let $K_Y$ be the induced canonical divisor on $Y$. 
Denote by $\p^*K_X$ Mumford's \emph{numerical pull-back} of $K_X$ to $Y$, which is uniquely determined as a $\Q$-divisor by the conditions $\p_* (\p^* K_X) = K_X$ and 
$\p^* K_X \cdot C =0$ for any $\p$-exceptional curve $C$.
The  \emph{log-discrepancy} divisor is then defined by the relation
$A_{Y/X} := K_Y+E -\p^*K_X$. Recall that $X$ is (numerically) lc iff $A_{Y/X}\ge 0$ while $X$ is (numerically) klt iff $A_{Y/X}>0$ on the whole of $E$. 

Wahl's invariant  measures the degree of positivity of the log-discrepancy divisor. 
The positivity is here relative to the contraction morphism $Y \to X$, and it is thus natural to consider the \emph{relative Zariski decomposition} $A_{Y/X}=P+N$ in the sense 
of~\cite[p.~408]{Sak84}, where $N$ is the smallest effective $\p$-exceptional $\Q$-divisor such that $P=A_{Y/X}-N$ is $\p$-nef. Finally one sets:
\begin{equation}\label{eq:rel-intro}
\Vol(X,0) :=  - P^2 \in \Q_{\ge 0}~.
\end{equation}

\medskip

Two (related) difficulties arise in generalizing Wahl's construction to higher dimensions: first, one needs to introduce a notion of pull-back for Weil divisors; and second, one needs to find 
a replacement for the relative Zariski decomposition.  These problems have already been addressed in~\cite{dFH}, and in~\cite{BFJ,KuMa} respectively.  
Building on these works our first objective is to explain how these difficulties can be conveniently addressed using Shokurov's language of \emph{$b$-divisors}. In \S\S\ref{sec:bdiv}-\ref{sec:mult-ideal}, we define and study
the notion of nef Weil $b$-divisor in the general setting of a normal variety $X$.
This leads to the notion of nef envelope and relative Zariski decomposition as follows.

Let us recall some terminology. 
A \emph{Weil $b$-divisor} $W$ over $X$ is the data of Weil divisors $W_\p$ on all birational models $\p\colon X_\p\to X$ of $X$ that are compatible under push-forward. A \emph{Cartier $b$-divisor} $C$ is a Weil $b$-divisor for which there is a model $\p$ such that for every other model
$\p'$ dominating $\p$ the trace $C_{\p'}$ of $C$ on $X_{\p'}$ is the pull-back
of the trace $C_\p$ on $X_\p$; any $\p$ as above is called a determination of $C$. 
All the divisors we consider for the time being have $\R$-coefficients.

Now, suppose we are given a projective morphism $f\colon X\to S$.
A Cartier $b$-divisor $C$ is said to be \emph{nef} (relatively to $f$) if $C_\p$ is nef for one (hence any) determination $\p$ of $C$. Generalizing \cite{BFJ,KuMa} we say that a Weil $b$-divisor $W$ is nef iff there exists a net of nef Cartier $b$-divisors $C_n$ such that
the net $[(C_n)_\p]$ converges to $[W_\p]$ in the space 
$N^1(X_\p/S)$ of numerical classes over $S$. This is 
equivalent to say that $W_\p$ lies in the closed movable cone $\overline{\Mov}(X_\p/S)$ for all smooth models $X_\p$ (cf.~Lemma~\ref{lem:mov} below).

In \S\ref{sec:nef}, we prove that the following definitions make sense (under suitable conditions), and introduce the following
two notions of {\it nef envelopes}.
\begin{itemize}
\item
The nef envelope  $\Env_X(D)$ of a Weil divisor $D$ on $X$ is the largest
nef Weil $b$-divisor $Z$ that is both relatively nef over $X$ and satisfies $Z_X\le D$.
\item
The nef envelope  $\Env_\X(W)$ of a Weil $b$-divisor $W$ is the largest
nef Weil $b$-divisor $Z$ that is both relatively nef over $X$ and satisfies $Z\le W$.
\end{itemize}
In dimension two, nef envelopes recover the notions of numerical pull-back and
relative Zariski decomposition. Specifically, if $D$ is a divisor on a normal surface $X$ then the trace $\Env_X(D)_\p$  on a given model $X_\p$ coincides with the numerical pull-back of $D$ by $\p$, while if  $D$ is a divisor on a smooth model $X_\p$ over $X$, then the nef part of $D$ in its relative Zariski decomposition is given by $\Env_\X(\overline D)_\p$ where $\overline{D}$ is the Cartier $b$-divisor induced by $D$.

In higher dimensions $D\mapsto\Env_X(D)$ is non-linear in general, and $\Env_X(D)_\p$ coincides up to sign with the pull-back $\p^*D$ defined in~\cite{dFH}. It is however this approach via $b$-divisors and nef envelopes that brings to light the crucial positivity properties of the pull-back of Weil divisors.

\medskip

We are now in a position to generalize the log-discrepancy divisor and its relative Zariski decomposition. Given a canonical divisor $K_X$ on $X$, there is a unique canonical divisor $K_{X_\p}$, for each model $\p\colon X_\p \to X$, with the property that $\p_*K_{X_\p} = K_X$. 
Thus a choice of $K_X$ determines a \emph{canonical $b$-divisor} $K_\X$ over $X$.
 The \emph{log-discrepancy $b$-divisor} is then defined as $$A_{\X/X}:=K_\X+1_{\X/X}+\Env_X(-K_X)~,$$ where the trace of $1_{\X/X}$ in any model is equal to the reduced exceptional divisor over $X$. The log-discrepancy $b$-divisor
is exceptional over $X$ and 
 does not depend on the choice of $K_X$. Its coefficients are given by the (usual) log-discrepancies of $X$ when the latter is $\Q$-Gorenstein. 
  The role of the nef part of $A_{\X/X}$ in its relative Zariski decomposition is in turn played by the nef envelope $$P:=\Env_\X(A_{\X/X})~.$$ 

To generalize~\eqref{eq:rel-intro}, we now face the problem of defining the intersection product 
of nef $b$-divisors. This step is non-trivial.  The intersection of Cartier $b$-divisors 
is defined as their intersection in a common determination. However it cannot be extended to a multilinear intersection product on the space of \emph{Weil} $b$-divisors having reasonable continuity properties. As it turns out, it is nevertheless possible to extend it to a multilinear intersection pairing on \emph{nef} Weil $b$-divisors lying over a point $0\in X$. This is done following the approach of~\cite{BFJ}, in which multiplier ideals appear as a prominent tool.

\medskip

Assume from now on that $(X,0)$ is an $n$-dimensional \emph{isolated} normal singularity. For all (relatively) nef $b$-divisors $W_1, ... , W_n$ above $0$, we set:
$$
W_1\cdot...\cdot W_n:=\inf\{C_1\cdot...\cdot C_n \mid C_j\text{ nef Cartier},\, C_j\ge W_j\}
\in  [-\infty,0]~.$$ 
To develop a reasonable calculus of these intersection numbers, 
\emph{additivity} in each variable is a  desirable property.
We obtain this result as a consequence of  the fact that any nef envelope of a Cartier $b$-divisor is the \emph{decreasing} limit of a sequence of nef Cartier $b$-divisors $C_k$.

Let us explain how to get this crucial approximation property. The first observation is that  the nef envelope of a Cartier $b$-divisor $C$ is a limit of the graded sequence of ideals $\fra_m:=\O_X(mC)$, $m\ge 0$ (see \S\ref{sec:nef-env}).
For any fixed $c>0$, we use the general notion of (asymptotic) multiplier ideal $\J(X;\fra_\bullet^c)$ introduced in \cite{dFH} for any ambient variety $X$ with normal singularities. 
As was shown in \cite{dFH} this multiplier ideal can also be computed using {\it compatible boundaries}: namely, there exist effective $\Q$-boundaries $\Delta$ such that $\J(X;\fra_\bullet^ c)$ coincides with the standard (asymptotic) multiplier ideal $\J((X,\D);\fra_\bullet^ c)$ with respect  to the pair $(X,\D)$.

This connection enables us to make use of a recent result of Takagi~\cite{Tak3}, which extends the usual subadditivity property of multiplier ideals~\cite{DEL} to multiplier ideals with respect to a pair $(X,\D)$, up to an (inevitable) error term involving $\D$ and the Jacobian ideal of $X$. The approximation we are looking for then follows by taking the nef Cartier $b$-divisor $C_k$ associated to $\J(X;\fra_\bullet^k)$.

Now that we have defined the intersection product of nef Weil $b$-divisors, we can come back to the definition of the volume. We set
$$
\Vol(X,0) := - \Env_\X(A_{\X/X})^n,
$$
which is shown to be finite (and non-negative). Once the volume is defined, the properties stated in Theorem~A follow smoothly
from transformation laws of envelopes under finite morphisms, see Proposition~\ref{prop:pull-env}.
$$
***
$$

The volume as defined above relates to other kind of invariants that were previously defined 
and are connected to growth rate of pluricanonical forms. 
 
In the $2$-dimensional case, we first note that the definition~\eqref{eq:rel-intro}
admits an equivalent formulation in terms of the growth rate of a certain quotient of sections. It was indeed shown in \cite{Wahl} that if $X$ is a surface then
$$
\dim\left(H^0(X\setminus \{0\},mK_X)/H^0(Y,m(K_Y+E))\right)=\frac{m^2}{2}\Vol(X,0)+o(m^2)~,
$$
where the left-hand side is independent of the choice of $Y$ and is equal by definition to the $m$-th \emph{log-plurigenus} $\lambda_m(X,0)$ in the sense of Morales \cite{Mora}, a notion which makes sense in all dimensions. 

In line with this point of view M.~Fulger~\cite{Fulg} has recently considered the following invariant of an isolated singularity $(X,0)$: 
$$
\Vol_F(X,0) : = \limsup_m \frac{n!}{m^n} \dim\left(H^0(X\setminus \{0\},mK_X)/H^0(Y,m(K_Y+E))\right).
$$
It measures by definition the growth rate of $\l_m(X,0)$, or equivalently that of Watanabe's \emph{$L^2$-plurigenera} $\d_m(X,0)$~\cite{Wat1,Wat2}, and yields a finite number since 
$$
\d_m(X,0)=\l_m(X,0)+O(m^{n-1})=O(m ^n)
$$ 
(see~\cite{Ish2}, which contains a thorough introduction to these notions, and \S\ref{sec:plurigenera} below). 

The notion of volume considered by Fulger also behaves well under finite morphisms, and the analog of Theorem~A holds true. Moreover, in contrast to our volume, $\Vol_F(X,0)$ is more accessible to explicit computations. On the other hand, our volume $\Vol(X,0)$ relates more closely
to lc singularities (see question~(b) below). 

Fulger explores in \cite{Fulg} how the two approaches compare to one another, proving
that $\Vol(X,0) \ge \Vol_F(X,0)$ for any isolated normal singularity $(X,0)$.
Equality holds when $X$ is $\Q$-Gorenstein, but can fail otherwise (cf.~Proposition~\ref{p:vol-Q-gor} and~Example~\ref{eg:counterfulger}). 

In general these volumes can take irrational values. 
In \cite{Urb} Urbinati constructs examples where the log-discrepancy takes irrational values, 
and in \cite{Fulg} Fulger shows that similar examples have irrational volumes $\Vol(X,0)$ and $\Vol_F(X,0)$. 
$$
***
$$
In the two dimensional case, we know by the work of Wahl \cite{Wahl} that the volume is a topological invariant of the link of the singularity and that its vanishing  characterizes log canonical singularities. Furthermore,  
Ganter~\cite{ganter} has shown that there is a uniform lower bound to
the volume of a normal Gorenstein surface singularity with positive volume. 
An example brought to our attention by Koll\'ar shows that the first property fails 
in higher dimensions: in general the volume of a normal isolated singularity is not a 
topological invariant of the singularity (cf. Example~\ref{eg:not-top-inv}).
The following questions remain open:
\begin{enumerate}
\item
Does there exists a \emph{positive} lower bound, only depending on the dimension, 
for the volume of isolated Gorenstein singularities with positive volume?
\item
Is it true that $\Vol(X,0) =0$ implies the existence of an effective $\Q$-boundary $\D$ such that the pair $(X,\D)$ is log-canonical? (the converse being easily shown). 
\end{enumerate}
It is to be noted that~(b) fails with $\Vol_F(X,0)$ in place of $\Vol(X,0)$ (cf. Example~\ref{eg:counterfulger}). 
$$
***
$$
The plan of our paper is the following.  In the first four sections, we work over 
a normal algebraic variety. \S\ref{sec:bdiv} contains basics on $b$-divisors. 
The notion of envelope
is analyzed in detail in \S\ref{sec:nef}. In this section we also formalize a measure of the failure of a Weil divisor to be Cartier in terms of certain {\it defect ideals}, which are related to the notion of compatible boundary.
In \S\ref{sec:mult-ideal} we turn to the definition of the log-discrepancy $b$-divisor and of multiplier ideals. The key result of this section is the subadditivity theorem (Theorem~\ref{thm:approx}) that we deduce from Takagi's work. 

The rest of the paper deals with normal isolated singularities. We define the volume of such a singularity and prove Theorem~A (i) and (iii) in \S\ref{sec:isolated}. In \S\ref{sec:compare} we  complete the proof of Theorem~A, and compare our notion with the approaches via plurigenera and Fulger's work.
Finally \S\ref{sec:endo} focuses on endomorphisms, and contains a proof of Theorem~B and~C.
$$
***
$$
\noindent {\bf Acknowledgements}
We would like to thank S. Takagi and M. Fulger for sharing with us preliminary versions of their work. We are also very grateful to Fulger for many precious comments, 
to J. Koll\'ar for bringing~\cite[Chapter~2, Example~55]{Kol11} to our attention, 
and to S.~Cacciola for pointing out a mistake in a previous formulation of
Theorem~\ref{Tak3}.
We would like to thank M.~Jonsson for allowing us to include a proof of 
Theorem~\ref{thm:increasing}, a result that will appear in a more general
form in \cite{BFJ11}.
This project started during the stay of the second author in Paris which was sponsored by the Fondation Sciences Math\'ematiques de Paris and the Institut de Math\'ematiques de Jussieu. It was completed when the first and last authors  visited the University of Utah in Salt Lake City.
We thank warmly all these institutions for their hospitality.
We would lake to thank the referee for a very careful reading and many precious comments
and suggestions.

%
%%%%%%%%%%%%%%%%%%%%%%
%

\section{Shokurov's $b$-divisors}\label{sec:bdiv}

In this section $X$ denotes a normal variety defined over an algebraically closed field of characteristic $0$ and we set $n:=\dim X$. The goal of this section is to gather general properties of Shokurov's $b$-divisors over $X$, for which \cite{Isk} and \cite{Corti} constitute general references. Proposition \ref{prop:pushcar} seems to be new.

\subsection{The Riemann-Zariski space}
The set of all proper birational morphisms $\p\colon X_\p\to X$ modulo isomorphism is (partially) ordered by $\p'\ge\p$ iff $\p'$ factors through $\p$, and the order is inductive (i.e.~any two proper birational morphisms to $X$ can be dominated by a third one). 
For short, we will refer to $X_\p$, or $\p$, as a \emph{model} over $X$.
The \emph{Riemann-Zariski space} of $X$ is defined as  the projective limit 
$$\X = \underleftarrow\lim_\p X_\p,$$ 
taken in the category of locally ringed topological spaces, each $X_\p$ being viewed as a scheme with its Zariski topology (note that $\X$ itself is not a scheme anymore). 

As a topological space $\X$ may alternatively be viewed as the set of all valuation subrings $V\subset k(X)$ with non-empty center on $X$, endowed with the Krull--Zariski topology. Indeed given a Krull valuation $V$ the center $c_\pi(V)$ of $V$ on $X_\pi$ is non-empty for each $\p$ by the valuative criterion for properness, and the collection of all scheme-theoretic points $c_\p(V)$ defines a point in $c(V)$ in $\X$. By \cite[p.122 Theorem 41]{ZS} the mapping $V\mapsto c(V)$ so defined is a homeomorphism. 

\subsection{Divisors on the Riemann-Zariski space}
\label{ss:1.2}
Following Shokurov we define the group of {\it Weil $b$-divisors} over $X$ (where $b$ stands for birational) as
$$
\Div(\X):= \underleftarrow{\lim}_\p \Div(X_\p)
$$
where $\Div(X_\p)$ denotes the group of Weil divisors of $X_\p$
and the limit is taken with respect to the push-forward maps
$\Div(X_{\p'}) \to \Div(X_\p)$, which are defined whenever $\p' \ge \p$. 
It can alternatively be thought of as the group of Weil divisors on the Riemann-Zariski space $\X$ (hence the notation). 

The group of {\it Cartier $b$-divisors} over $X$ is in turn defined as
$$
\CDiv(\X) := \underrightarrow{\lim}_\p \CDiv(X_\p)
$$
with $\CDiv(X_\p)$ denoting the group of Cartier divisors of $X_\p$. 
Here the limit is taken with respect to the pull-back maps
$\CDiv(X_{\p}) \to \CDiv(X_{\p'})$, which are defined whenever $\p' \ge \p$. 
One can easily check that 
$$\CDiv(\X)=H^0(\X,\M_\X^*/\O_\X^*)$$
is indeed the group of Cartier divisors of the locally ringed space $\X$. 

There is an injection $\CDiv(\X) \inj \Div(\X)$ determined by the cycle maps on birational models $X_\p$.

An element of $\Div_\R(\X):=\Div(\X)\otimes\R$ (resp.~$\CDiv_\R(\X):=\CDiv(\X)\otimes\R$) will be called an $\R$-Weil $b$-divisor (resp.~$\R$-Cartier $b$-divisor), and similarly with $\Q$ in place of $\R$. 
The space $\Div_\R(\X)$ is naturally isomorphic to the projective limit of the spaces
$\Div_\R(X_\p)$, and $\CDiv_\R(\X)$ is naturally isomorphic to the direct limit of the spaces
$\CDiv_\R(X_\p)$.

Let us now interpret  these definitions in more concrete terms.
A Weil divisor $W$ on $\X$ consists of a family of Weil divisors
$W_\p\in\Div(X_\p)$ that are compatible under push-forward, i.e.~such that $W_{\p} = \m_* W_{\p'}$ whenever $\p'$ factors through
a morphism $\m \colon X_{\p'} \to X_{\p}$. We say that
$W_\p$ (also denoted by $W_{X_\p}$) is the \emph{trace} (or {\it incarnation} as in \cite{BFJ}) 
of $W$ on the model $X_\p$.
By contrast, a Cartier divisor $C$ on $\X$ is determined
by its trace on a high enough model, i.e.~there exists $\p$ such that $C_{\p'} = \m^*C_\p$ for every $\p' \ge \p$,
where $\m \colon X_{\p'} \to X_{\p}$ is the induced morphism. We shall say that $C$ is \emph{determined on $X_\p$} (or \emph{by $\p$}).\newline

Weil $b$-divisors can also be interpreted as certain functions on the set of \emph{divisorial valuations} of $X$. Recall first that a divisorial valuation of $X$ is a rank 1 valuation of transcendence degree $\dim X-1$ of the function field $k(X)$, whose center on $X$ is non-empty. By a classical result of Zariski (see e.g.~\cite[Lemma~2.45]{KM}) the divisorial valuations on $X$ are exactly those of the form $\nu=t\ord_E$ where $t\in\R_+^*$ and $E$ is a prime divisor on some birational model $X_\p$ over $X$.  

Given an $\R$-Weil $b$-divisor $W$ over $X$ we can then define $(t\ord_E)(W)$ as $t$ times the coefficient of $E$ in $W_\p$. 

\begin{lem}
Setting $g_W(\nu):=\nu(W)$ yields an identification $W\mapsto g_W$ between $\Div_\R(\X)$ and the space of all real-valued $1$-homogeneous functions $g$ on the set of divisorial valuations of $X$ satisfying the following finiteness property: the set of prime divisors $E\subset X$ (or equivalently on $X_\p$ for any given $\p$) such that $g(\ord_E)\neq 0$ is finite. 
\end{lem}

The topology of pointwise convergence therefore induces a \emph{topology of coefficient-wise} convergence on $\Div_\R(\X)$, for which $\lim_j W_j=W$ iff $\lim_j\ord_E(W_j)=\ord_E(W)$ for each prime divisor $E$ over $X$.

\subsection{Examples of $b$-divisors}
We introduce the main types of $b$-divisors we shall consider. 

\begin{eg} The choice of a non-zero rational form $\omega$ of top degree on $X$ induces a \emph{canonical $b$-divisor} $K_\X$ whose trace on $X_\p$ is equal to the canonical divisor determined by $\omega$ on $X_\p$.
\end{eg} 

\begin{eg} A Cartier divisor $D$ on a given model $X_\p$ induces a Cartier $b$-divisor $\ov D$, its \emph{pull-back} to $\X$. It is simply defined by pulling-back $D$ to all models dominating $X_\p$
and then by pushing-forward on all other models. By definition all Cartier $b$-divisors are actually obtained this way. 
\end{eg}

%\begin{eg} A Weil divisor $D$ on a given model $X_\p$ induces a Weil $b$-divisor on $\X$ %called
%the \emph{proper transform} of $D$. It is simply defined by considering on each model
%the proper transform of $D$. 
%\end{eg}

\begin{eg} Given a coherent fractional ideal sheaf $\fra$ on $X$ we denote by $Z(\fra)$ the Cartier $b$-divisor determined on the normalized blow-up $X_\p$ of $X$ along $\fra$ by
$$\fra\cdot\O_{X_\p}=\O_{X_\p}(Z(\fra)_\p).$$
In particular we have $Z(f)_\p=-\p^*\dv(f)$ when $f$ is a rational function on $X$. Note that with this convention $Z(\fra)$ is anti-effective when $\fra$ is an actual ideal sheaf. 
\end{eg}

For any Weil $b$-divisor we write $Z \ge 0$ if $Z_\p$ is an effective divisor for every $\p$.
We record the following easy properties. 

\begin{lem}\label{lem:Zfrac} Let $\fra,\frb$ be two coherent fractional ideal sheaves on $X$.
\begin{itemize}
\item $Z(\fra) \le Z(\frb)$ whenever $\fra \subset \frb$.
\item $Z(\fra\cdot\frb) = Z(\fra) + Z(\frb)$.  
\item $Z(\fra + \frb) = \max\{Z(\fra),Z(\frb)\}$, where the maximum is
defined coefficient-wise.
\item $Z(\fra)=Z(\frb)$ iff the integral closures of $\fra$ and $\frb$ are equal.
\end{itemize}
\end{lem}

\begin{rmk} Given an ideal sheaf $\fra$ and a positive number $s>0$ we set $Z(\fra^s):=s Z(\fra)$. Then, by definition, we have $Z(\fra^s)=Z(\frb^t)$ iff the `$\R$-ideals' $\fra^s$ and $\frb^t$ are `valuatively' equivalent in the sense of Kawakita \cite{Kaw}. 
\end{rmk}

\begin{defi} Let $W$ be an $\R$-Weil $b$-divisor over $X$. We denote by $\O_X(W)$ the fractional ideal sheaf of $X$ whose sections on an open set $U\subset X$ are the rational functions $f$ such that $Z(f)\le W$ over $U$. 
\end{defi}
We emphasize that the sheaf of $\O_X$-modules $\O_X(W)$ is \emph{not}  coherent in general, since we are imposing infinitely many (even uncountably many) conditions on $f$ (compare \cite{Isk}). Note that $\p_*\O_{X_\p}(W_\p)\subset\tau_*\O_{X_\tau}(W_\tau)$ whenever $\pi\ge\tau$ and 
$$
\O_X(W)=\bigcap_\p\p_*\O_{X_\p}(W_\p).
$$ 
However if $C$ is an $\R$-Cartier $b$-divisor then we have $\O_X(C)=\p_*\O_{X_\p}(C_\p)$ for each determination $\p$ of $C$, and $\O_X(C)$ is in particular coherent in that case. 

Cartier $b$-divisors associated with coherent fractional ideal sheaves can be characterized as follows:

\begin{lem}\label{lem:difference}
A Cartier $b$-divisor $C\in\CDiv(\X)$ is of the form $Z(\fra)$ for some coherent fractional ideal sheaf $\fra$ on $X$ iff $C$ is relatively globally generated over $X$. 

In particular the Cartier divisors $Z(\fra)$ with $\fra$ ranging over all coherent (fractional) ideal sheaves of $X$ generate $\CDiv(\X)$ as a group.
\end{lem}
Here we say that $C$ is relatively globally generated over $X$ iff so is $C_\pi$ for one (hence any) determination $\p$ of $C$.

\begin{proof} Let $C$ be a Cartier $b$-divisor determined by $\p$. To say that $C$ is relatively globally generated over $X$ means by definition that the evaluation map
$$\p^*\p_*\O_{X_\p}(C_\p)\to\O_{X_\p}(C_\p)$$
is surjective. If this is the case we thus see that $C=Z(\fra)$ with $\fra:=\p_*\O_{X_\p}(C_\p)=\O_X(C)$, while the converse direction is equally clear. The second assertion now follows from the fact that any Cartier divisor on a given model $X_\p$ can be written as a difference of two $\p$-very ample (hence $\p$-globally generated) Cartier divisors. 
\end{proof}

\subsection{Numerical classes of $b$-divisors} 
Let $X\to S$ be a projective morphism. Recall that the space of codimension one
relative numerical classes $N^1(X/S)$ is the vector space of $\R$-Cartier divisors modulo
those divisors $D$ for which $D\.C = 0$ for every irreducible curve $C$
that is mapped to a point in $S$. One can put together these spaces and 
define the space of \emph{$1$-codimensional numerical classes} of $\X$ over $S$ by
$$
N^1(\X/S):=\underrightarrow{\lim}_\p N^1(X_\p/S)
$$
where the maps are given by pulling-back. We define in turn the space of \emph{$(n-1)$-dimensional numerical classes} of $\X$ over $S$ by
$$
N_{n-1}(\X/S):=\underleftarrow{\lim}_\p N^1(X_\p/S)
$$
where the maps are given by pushing-forward and $\p$ now runs over all \emph{smooth} (or at least $\Q$-factorial) birational models of $X$ -- so that the push-forward map $N^1(X_{\p'}/S)\to N^1(X_\p/S)$ is well-defined for $\p'\ge\p$. 
 
Each $N^1(X_\pi/S)$ is a finite dimensional $\R$-vector space and we endow $N^1(\X/S)$ and $N_{n-1}(\X/S)$ with their natural inductive and projective limit topologies respectively. 

\begin{lem}
The cycle maps induce a natural continuous injection $N^1(\X/S) \to N_{n-1}(\X/S)$
with dense image.
\end{lem}

\begin{proof}
Just as in the case of Cartier and Weil $b$-divisors described in Subsection~\ref{ss:1.2},
any class $\b$ in $N^1(X_\p/S)$ can be identified to the class in $N_{n-1}(\X/S)$
determined by pulling back $\b$ on all higher models. We thus 
have natural continuous maps $N^1(X_\p/S) \to N_{n-1}(\X/S)$
which induce a continuous injective map $N^1(\X/S)\to N_{n-1}(\X/S)$.
It follows by the definition of the projective limit topology that this map has dense image, 
since for any class $\a \in N_{n-1}(\X/S)$ the net determined by its traces
$\a_\p \in N^1(X_\p/S)$, viewed as elements of $N_{n-1}(\X/S)$ as described before, 
converges to $\a$.
\end{proof}

There are also natural surjections $\CDiv_\R(\X)\to N^1(\X/S)$ and $\Div_\R(\X)\to N_{n-1}(\X/S)$, but one should be careful that the latter map is \emph{not} continuous with respect to coefficient-wise convergence in general. 

\begin{eg} Consider an infinite sequence $C_j$ of $(-1)$-curves on $X=\P^2$ blown-up at 9 points. We then have $C_j\to 0$ coefficient-wise but the numerical classes $[C_j]\in N^1(X)$ do not tend to zero since $C_j^2=-1$ for each $j$. 
\end{eg}

\begin{lem} Let $\p\colon X_\p\to X$ be a birational model of $X$ and let $\a\in N^1(X_\p/X)$. Then there exists at most one $\p$-exceptional $\R$-Cartier divisor $D$ on $X_\p$ whose numerical class is equal to $\a$. 
\end{lem}
\begin{proof} Let $D$ be a $\p$-exceptional and $\p$-numerically trivial $\R$-Cartier divisor. We are to show that $D=0$. Upon pulling-back $D$ to a higher birational model, we may assume that $\p$ is the normalized blow-up of $X$ along a subscheme of codimension at least two. If we denote by $E_j$ the $\p$-exceptional divisors we then have on the one hand $D=\sum_j d_j E_j$ and on the other hand there exists positive integers $a_j$ such that $F:=\sum_j a_j E_j$ is $\p$-antiample. Now set $t:=\max_j d_j/a_j$. If we assume by contradiction that $D\neq 0$ then upon possibly replacing $D$ by $-D$ we may assume that $t>0$. Now $tF-D$ is effective and there exists $j$ such that $E_j$ is not contained in its support. If $C\subset E_j$ is a general curve in a fiber of $\p$ we then have $(tF-D)\cdot C\ge 0$ since $C$ is not contained in the support of the effective divisor $tF-D$, which contradicts the fact that $D-tF$ is $\p$-ample.
\end{proof}

Even assuming that $X_\p$ is smooth, it is not true in general that any class $\a\in N^1(X_\p/X)$ can be represented by a $\p$-exceptional $\R$-divisor (since $\p$ might for instance be small, i.e.~without any $\p$-exceptional divisor). It is however true when $X$ is $\Q$-factorial, and for any normal $X$ when $\dim X=2$ thanks to Mumford's numerical pull-back. 

Using these remarks we may now prove the following simple lemma which enables to circumvent the discontinuity of the quotient map $\Div_\R(\X)\to N_{n-1}(\X/S)$.

\begin{lem}\label{lem:repres} ~
\begin{enumerate}
\item Let $W_j$ be a sequence (or net) of $\R$-Weil $b$-divisors which converges to an $\R$-Weil $b$-divisor $W$ coefficient-wise. If there exists a fixed finite dimensional vector space $V$ of $\R$-Weil divisors on $X$ such that $W_{j,X}\in V$ for all $j$ then $[W_j]\to[W]$ in $N_{n-1}(\X/S)$. 
\item Let conversely $\a_j\to \a$ be a convergent sequence (or net) in $N_{n-1}(\X/S)$. Then there exist representatives $W_j,W\in\Div_\R(\X)$ of $\a_j$ and $\a$ respectively and a finite dimensional vector space $V$ of $\R$-Weil divisors on $X$ such that 
\begin{itemize} 
\item $W_j\to W$ coefficient-wise. 
\item $W_{j,X}\in V$ for all $j$. 
\end{itemize}
If $\a_j\in N^1(\X/S)$ then $W_j$ can be chosen to be $\R$-Cartier. 
\end{enumerate}
\end{lem}
\begin{proof} For each smooth model $\p$ the existence of $V$ yields a finite dimensional space $V_\p$ of $\R$-divisors on $X_\p$ such that 
$W_{j,\p}\in V_\p$ for all $j$. The natural linear map $V_\p\to N^1(X_\p/S)$ is of course continuous since both spaces are finite dimensional, and it follows that $[W_{j,\p}]\to[W_\p]$ in $N^1(X_\p/S)$ for each smooth model. Since smooth models are cofinal in the family of all models we conclude as desired that $[W_j]\to[W]$ in $N_{n-1}(\X/S)$. 

We now consider the converse. Let $X_\p$ be a fixed smooth model of $X$. For each $j$, $\a_j-\overline\a_{j,\p}$ (resp.~$\a-\overline\a_\p$) is exceptional over $X_\p$. By the above remarks it is thus \emph{uniquely} represented by an $\R$-Weil $b$-divisor $Z_j$ (resp.~$Z$) that is exceptional over $X_\p$. Since $(\a_j-\overline\a_{j,\p})_{\p'}$ converges to $(\a-\overline\a_\p)_{\p'}$ in $N^1(X_{\p'}/X_\p)$ for each $\p'\ge\p$ it follows by uniqueness of $Z_j$ that $Z_j\to Z$ coefficient-wise. 

On the other hand since $N^1(X_\p/S)$ is finite dimensional there exists a finite dimensional $\R$-vector space $V$ of $\R$-divisors on $X_\p$ such that $V\to N^1(X_\p/X)$ is surjective. This map is therefore open and we may thus find representatives $C_j\in V$ of $\a_{j,\p}$ converging to a representative $C\in V$ of $\a_\p$. Setting $W_j:=Z_j+\overline C_j$ concludes the proof. 
\end{proof}

\subsection{Functoriality}
If $\phi\colon X\to Y$ is any morphism between two normal varieties, then it is immediate to see that pulling back induces a homomorphism $\phi^*\colon \CDiv(\Y)\to\CDiv(\X)$ in a functorial way. 

Assume furthermore that $\phi\colon X\to Y$ is proper, surjective and generically finite. In this case pushing forward induces a homomorphism
$$
\phi_*\colon \Div(\X)\to\Div(\Y),
$$
and the homomorphism $\phi^*\colon \CDiv(\Y)\to\CDiv(\X)$ extends in a natural way to a homomorphism 
$$
\f^* \colon \Div(\Y)\to\Div(\X). 
$$

Before going through the constructions of these homomorphisms, we recall the following property.

\begin{lem}\label{lem:surjdiv} Let $\phi\colon X\to Y$ be a proper, surjective and generically finite morphism of normal varieties.
Every divisorial valuation $\nu$ on $X$ induces, by restriction via the
field extension $\f^*\colon \C(Y) \inj \C(X)$, a divisorial 
valuation $\phi_*\nu$ on $Y$ that defined by
$$
(\phi_*\nu)(f):=\nu(f\circ\phi).
$$ 
The correspondence $\nu\mapsto\phi_*\nu$ defines a surjective map with finite fibers 
from the set of divisorial valuations on $X$ to the set of divisorial valuations on $Y$.
\end{lem}

\begin{proof}Ê
If $\nu$ is a divisorial valuation on $X$ then $\f_*\n$ is a divisorial valuation on $Y$ since the restriction of the valuation ring of $\nu$ to $\C(Y)$ has transcendence degree $\dim Y-1$ by \cite[VI.6, Corollary~1]{ZS}.
The assertion is that, if $\n'$ is a divisorial valuation on $Y$, 
then there exists a nonzero finite number
of divisorial valuations $\n_1,\dots,\n_r$ on $X$ that restrict to $\n'$. 
Geometrically, if $\n' = t\ord_F$ where $F$ is a prime divisor on some model $Y'$ over $Y$ and $t > 0$, 
then the valuations $\n_i$ are constructed by picking model $X'$ over $X$ such that $\f$ lifts to a well-defined morphism $\f' \colon X' \to Y'$.
If $E_1,\dots,E_r$ are the irreducible components of $(\f')^*F$ such that $\f'(E_i)= F$, 
then the associated valuations $\ord_{E_i}$ restrict to a multiple of $\ord_F$ on $\C(Y)$.
Up to rescaling, these are the only divisorial valuations restricting to $\ord_F$ 
since any divisorial valuation on $X$ with non-divisorial center in $X'$ restricts 
to a divisorial valuation on $Y$ with non-divisorial center in $Y'$.
\end{proof} 

We then define $\phi_*\colon \Div(\X)\to\Div(\Y)$ and $\phi^*\colon \Div(\Y)\to\Div(\X)$ in the following way.
If $W \in\Div(\X)$, then $\f_* W$ is characterized by the condition that
$$
\ord_F(\f_*W) = \sum_i \ord_F((\f')_*E_i) \. \ord_{E_i}(W).
$$
for any prime divisor $F$ over $Y$.
Here we are using the notation as in the proof of Lemma~\ref{lem:surjdiv}, so that
$F$ is a divisor on a model $Y'$ over $Y$, $X'$ is a model over $X$ such that the map $\f' \colon X' \to X$ induced by $\f$ is a morphism, and the $E_i$ are the irreducible components of $(\f')^*F$ dominating $F$. 
It follows by the lemma that the sum is finite.
Note also that on any model $Y'$ the coefficient $\ord_F(\f_*W)$ can be nonzero only for finitely many prime divisors $F$ on a model $X'$, so that $\f_*W$ does define a Weil $b$-divisor over $Y$.  

Regarding the pull-back, if $W \in \Div(\Y)$, then $\f^*W$ is characterized by the condition that
$$
\ord_E(\f^*W) = (\f_*\ord_E)(W)
$$
for every prime divisor $E$ over $X$. 
This is indeed a Weil $b$-divisor since each prime divisor $E$ on $X$ such that $(\phi_*\ord_E)(W)\neq 0$ is either mapped to a prime divisor $F$ on $Y$ such that $\ord_F(W)\neq 0$ or is contracted by $\phi$, so that the set of all such prime divisors $E$ is appearing on any model $X'$ over $X$ finite by Lemma~\ref{lem:surjdiv}.

%Given a prime divisor $E$ over $X$ we thus have $\f_* \ord_E = b\ord_F$ for some positive number $b$ and some prime divisor $F$ over $Y$. The coefficient $b$ is actually an integer and can be described as follows: there exist models $X'\to X$ and $Y'\to Y$ such that $E$ (resp.~$F$) is a prime divisor on $X'$ (resp.~$Y'$) and such that the rational lift $\phi'\colon X'\dashrightarrow Y'$ of $\phi$ sends the generic point of $E$ to that of $F$. We then have $b=\ord_E((\phi')^*F)$ (the pull-back is well-defined since it is only considered at the generic point of $F$ and $Y'$ is regular in codimension $1$). 

\begin{prop}\label{prop:pushcar} Let $\phi\colon X\to Y$ be a proper, surjective, generically finite morphism. Then $\phi_*\CDiv(\X)\subset\CDiv(\Y)$.
\end{prop}

\begin{proof} The assertion is obvious when $\phi$ is birational because we are just shifting models in that case. Using the Stein factorization of $\phi$ we may thus assume that $\phi$ is finite (and still proper and surjective). By Lemma \ref{lem:difference} it is then enough to show that for every coherent fractional ideal sheaf $\fra$ on $X$ there exists a coherent fractional ideal sheaf $\frb$ on $Y$ such that $\phi_*Z(\fra)=Z(\frb)$. In fact we claim that 
\begin{equation}\label{equ:norm}
\phi_*Z(\fra)=Z(N_{X/Y}(\fra))
\end{equation}
where $N_{X/Y}(\fra)$ denotes the image of $\fra$ under the \emph{norm homomorphism} (compare \cite[D\'efinition 21.5.5]{EGA4}). 

More precisely pick an affine chart $U\subset Y$. Since the restriction $\f^{-1}(U) \to U$ is finite,
$\f^{-1}(U)$ is affine and $\fra$ is thus generated by its global sections $g$ on $\f^{-1}(U)$.
For each such $g$ its norm is defined by setting
$$
N_{X/Y}(g)(x) = \prod_{\f(y) =x} g(y)
$$
for every smooth point $x \in U$ over which $\f$ is \'etale and by extending it to a regular function on $U$ by normality. We then define $N_{X/Y}(\fra)(U)$ 
as the $\O_U$-module generated by all $N_{X/Y}(g)$ with $g$ as above.

Let us now prove (\ref{equ:norm}). Pick a prime divisor $F$ on a model $Y'$ over $Y$ and choose a birational model $X'$ over $X$ such that $\f$ lifts to a morphism
$\f' \colon X' \to Y'$. Note that $\f'$ 
is proper and generically finite. Let $E_1,\dots,E_r$ be the prime divisors of $X'$ dominating $F$, so that $(\f')_* E_i = c_i F$ for some positive integer $c_i$. 
Then we have 
$$
\ord_F (\f_* Z(\fra)) = \sum_i c_i \ord_{E_i}(Z(\fra))=-\sum_i c_i\ord_{E_i}(\fra) 
$$
by definition of $\phi_*$.  On the other hand, let $V\subset Y'$ be an affine chart containing a point of $F$.
The ideal sheaf $N_{X/Y}(\fra)\cdot\O_{Y'}$ 
is generated, over $V$, by the functions $N_{X'/Y'}(g)$ where $g$ ranges 
over all global sections
of $\fra\cdot\O_{X'}$ on $(\f')^{-1}(V)$. We have 
$$
\ord_F (N_{X'/Y'}(g)) = \sum c_i\ord_{E_i}(g)
$$
hence
$$
\ord_F (N_{X/Y}(\fra))=\min \big\{ \ord_F (N_{X'/Y'}(g)),\,g\in H^0((\f')^{-1}(V),\fra\cdot\O_{X'}) \big\} 
$$
$$
=\min\big\{\sum c_i \ord_{E_i}(f),\,f\in\fra\big\}
$$
which proves the claim since we have $\ord_{E_i}(f)=\ord_{E_i}(\fra)$ for each $i$ if $f\in\fra$ is a general element. 
\end{proof}

\begin{prop}\label{prop:push-cartier}
Suppose $\f\colon  X \to Y$ is a proper, surjective, generically finite morphism of normal varieties, and let $e(\f) \in\N^*$ be its degree.
Then we have 
$$
\f_* \f^* W = e(\f)\, W
$$
for every $W\in\Div(\Y)$.
\end{prop}

\begin{proof}
Let $F$ be an arbitrary prime divisor over $Y$, and let $E_1,\dots,E_r$ be the prime divisors
over $X$ such that $\ord_{E_i}$ restricts to a multiple of $\ord_F$. 
Let $X' \to X$ and $Y' \to Y$ be models so that each $E_i$ is on $X'$ and $E$ is on $Y'$. 
As before, we can assume that $\f$ lifts to a morphism $\f' \colon X' \to Y'$. 
Let $c_i = \ord_F((\f')_*E_i)$. 
By definition of $\f_*$ and $\f^*$, we have 
$$
\ord_F(\f_*\f^*W) = \sum_i c_i \ord_{E_i}(\f^*W) = \sum_i c_i \ord_{E_i}(\f^*F) \ord_F(W)
= e(\f') \ord_F(W),
$$
where the last equality follows by projection formula. 
One concludes by observing that $e(\f') = e(\f)$. 
\end{proof}

%
%%%%%%%%%%%%%%%%%%%%%%
%

\section{Nef envelopes} \label{sec:nef}
In this section $X$ still denotes an arbitrary normal variety (over an algebraically closed field of characteristic zero). We reinterpret the pull-back construction of \cite{dFH} as a \emph{nef envelope}, which shows in particular that it coincides with Mumford's numerical pull-back on surfaces. Section \ref{sec:defect} introduces the \emph{defect ideal} of a Weil divisor, measuring its failure to be Cartier, and a precise description of the defect ideal is obtained. 

\subsection{Graded sequences and nef envelopes}\label{sec:nef-env}
Recall that $\fra_\bullet=(\fra_m)_{m \ge 0}$ is a \emph{graded sequence of fractional ideal sheaves} if $\fra_0=\O_X$, each $\fra_m$ is a coherent fractional ideal sheaf of $X$ and $\fra_k\.\fra_m \subset \fra_{k+m}$ for every $k,m$ (see \cite[Section~2.4]{Laz1}). 
We shall say that $\fra_\bullet$ has \emph{linearly bounded denominators} if there exists a (fixed) Weil divisor $D$ on $X$ such that $\O_X(mD)\cdot\fra_m\subset\O_X$ for all $m$. 

Let us first attach an $\R$-Weil  $b$-divisor to any graded sequence of ideal sheaves with linearly bounded denominators:
\begin{prop}\label{prop:graded-sequence}
Suppose that $\fra_\bullet = (\fra_m)_{m \ge 0}$ is a graded sequence of
fractional ideals sheaves $\fra_m$ with linearly bounded denominators. Then we have 
$$\tfrac 1lZ(\fra_l)\le \tfrac 1mZ(\fra_m)$$
for every $m$ divisible by $l$ and the sequence $\tfrac 1m Z(\fra_m)$ converges coefficient-wise to an $\R$-Weil $b$-divisor. We shall write 
$$
Z(\fra_\bullet) := \lim_m \tfrac 1m Z(\fra_m).
$$
\end{prop}

\begin{proof} All this follows from the super-additivity property
$$Z(\fra_m) + Z(\fra_n) \le  Z(\fra_{m+n})$$
since the condition that $\fra_\bullet$ has linearly bounded denominators guarantees that the sequence $\tfrac 1m\ord_E Z(\fra_m)$ is bounded below for each prime divisor $E$ over $X$ and even identically zero for all but finitely many prime divisors $E$ on $X$. 
\end{proof}

\begin{lem}\label{lem:fingen} Let $\fra_\bullet$ be a graded sequence of fractional ideal sheaves on $X$ with linearly bounded denominators. Then we have $Z(\fra_\bullet)=\tfrac{1}{m_0} Z(\fra_{m_0})$ for some $m_0$ iff the graded $\O_X$-algebra $\bigoplus_{m\ge 0}\overline{\fra_m}$ of integral closures is finitely generated. 
\end{lem}
\begin{proof} Since $Z(\fra_m)$ only depends on $\overline{\fra_m}$ (cf.~Lemma \ref{lem:Zfrac}), we may assume to begin with that every $\fra_m$ is integrally closed. Assume first that the graded algebra is finitely generated, so that there exists $m_0\in\N$ such that $\fra_{km_0}=\fra_{m_0}^k$ for all $k\in\N$. Then $Z(\fra_{km_0})=kZ(\fra_{m_0})$, hence $Z(\fra_\bullet)=\tfrac{1}{m_0}Z(\fra_{m_0})$. Conversely, assume that $Z(\fra_\bullet)=\tfrac{1}{m_0} Z(\fra_{m_0})$ for a given $m_0$. By Proposition \ref{prop:graded-sequence} it follows that $Z(\fra_{km_0})=k Z(\fra_{m_0})$ for all $k$. Let $\p$ be the normalized blow-up of $X$ along $\fra_{m_0}$. We then have
$$
\fra_{km_0}=\overline{\fra_{km_0}}=\p_*\O_{X_\p}\left(k Z\left(\fra_{m_0}\right)_\p\right)
$$ 
for all $k$ (cf.~\cite[Proposition~9.6.6]{Laz1}). Since the graded algebra of (relative) global sections of multiples of any (relatively) globally generated line bundle is finitely generated, the fact that $Z(\fra_{m_0})_\p$ is $\p$-globally generated implies that the $\O_X$-algebra $\bigoplus_k\fra_{km_0}$
is finitely generated, hence so is its finite integral extension $\bigoplus_m\fra_m$. 
\end{proof}

\begin{defi} Let $D$ be an $\R$-Weil divisor on $X_\p$ for a given $\p$. The \emph{nef envelope} $\Env_\p(D)$ of $D$ is defined as the $\R$-Weil $b$-divisor associated with the graded sequence $\p_*\O_{X_\p}(mD)$, $m\ge 0$. 
When $\p$ is the identity we write $\Env_X$ for $\Env_\p$. 
\end{defi}

We shall see how this definition relates to relative Zariski decomposition and 
numerical pull-back in the surface case (see Theorem~\ref{thm:mumford}). 
A non-trivial toric example is worked out in Example~\ref{eg:counter}.

\begin{rmk} If $D$ is an $\R$-Weil divisor on $X$ then $-\Env_X(-D)_\p$ coincides by definition with $\p^*D$ in the sense of \cite[Definition 2.9]{dFH}. 
\end{rmk}

\begin{rmk}
We shall introduce later in Subsection~\ref{ss:2.3} a notion of nef envelope over $\X$
of a $b$-divisor $W$ (under some condition on $W$). 
The relation between the two notions of envelopes is explained in Remark~\ref{rmk:Env-comparison}.
\end{rmk}

\begin{prop}\label{prop:concave}
Let $D, D'$ be two $\R$-Weil divisors on a model $X_\p$. Then we have:\begin{itemize} 
\item $\Env_\p(D+D') \ge \Env_\p(D) + \Env_\p(D')$.
\item $\Env_\p(tD)=t\Env_\p(D)$ for each $t\in\R_+$
\end{itemize}
\end{prop}

\begin{proof}
For each $m\ge 0$ we have
$$
(\p_*\O_{X_\p}(mD))\cdot(\p_*\O_{X_\p}(mD')) \subset\p_*\O_{X_\p}(m (D+D'))
$$
whence the first point. 

In order to prove the second point we may assume that $D$ is effective (since we may add to $D$ the pull-back of an appropriate Cartier divisor of $X$ to make it effective). Now observe that $\Env_\p(mD)=m\Env_\p(D)$ for each positive integer $m$ since $\Env_\p(D)=\lim_k\tfrac 1k Z(\p_*\O_{X_\p}(kD))$, hence $\Env_\p(tD)=t\Env_\p(D)$ for each $t\in\Q_+^*$. On the other hand $D\mapsto\Env_\p(D)$ is obviously non-decreasing, so if we pick $t\in\R_+^*$ and approximate it from below and from above by rational numbers $s_j,t_j$ we get
$$
s_j\Env_\p(D)=\Env_\p(s_j D)\le \Env_\p(tD)\le\Env_\p(t_j D)=t_j\Env_\p(D)$$
hence the result.
\end{proof}

Linearity of nef envelopes fails in general. The obstruction to linearity will be studied in greater detail in Section \ref{sec:defect} (see also Example~\ref{eg:counter} and~\cite{dFH}).

\begin{cor}\label{cor:cont}
For every finite dimensional vector space $V$ of $\R$-Weil divisors on $X_\p$ and every divisorial valuation $\nu$ the map $D\mapsto\nu(\Env_\p(D))$ is continuous on $V$.  
\end{cor}
\begin{proof} Proposition \ref{prop:concave}  implies that $D\mapsto\nu(\Env_\p(D))$ is a concave function on $V$ and the result follows.
\end{proof}

\begin{prop}\label{prop:envx}
For every $\R$-Weil divisor $D$ on $X$ the trace $\left(\Env_X(D)\right)_X$ of $\Env_X(D)$ on $X$ coincides with $D$.
\end{prop}
\begin{proof} If $D$ is a Weil divisor on $X$ then we have $Z(\O_X(D))_X=D$. Indeed this means that $\ord_E\O_X(D)=-\ord_E D$ for each prime divisor $E$ of $X$, which holds true since $X$, being normal, is regular at the generic point of $E$. 

As a consequence we get $D=\left(\Env_X(D)\right)_X$ when $D$ is a $\Q$-Weil divisor on $X$, and the general case follows by density, using Corollary \ref{cor:cont}.
\end{proof}

\subsection{Variational characterization of nef envelopes}
Let $X\to S$ be a projective morphism. In the usual theory of $b$-divisors one says that an $\R$-Cartier $b$-divisor $C$ is relatively nef over $S$ (or $S$-nef for short) if $C_\p$ is $S$-nef for one (hence any) determination $\p$ of $C$. Following \cite{BFJ,KuMa} we extend this definition to arbitrary $\R$-Weil $b$-divisors:

\begin{defi} Let $X\to S$ be a projective morphism. We define $\Nef(\X/S)\subset N_{n-1}(\X/S)$ as the closed convex cone generated by all $S$-nef classes $\b\in N^1(\X/S)$, i.e.~all classes of $S$-nef $\R$-Cartier $b$-divisors. 
\end{defi}
Since the usual notion of nefness is preserved by pull-back, it is immediate to check that $S$-nef classes in the sense of the above definition are also preserved by pull-back. On the other hand nefness is in general not preserved under push-forward when $\dim X>2$, and the traces $W_\p$ of an $S$-nef $\R$-Weil $b$-divisor are therefore \emph{not} $S$-nef in general. 

Given a projective morphism $Y\to S$, the \emph{$S$-movable cone} $\overline{\mathrm{Mov}}(Y/S)\subset N^1(Y/S)$ is the closed convex cone $\overline{\mathrm{Mov}}(Y/S)$ generated by the numerical classes of all Cartier divisors $D$ on $Y$ whose $S$-base locus has codimension at least two. Recall that the $S$-base locus of a Cartier divisor $D$ on $Y$ is the cosupport
of the ideal sheaf obtained as the image of the natural evaluation map
$f^*f_*\O_Y(D) \otimes \O_Y(-D) \to \O_Y$.

We now have the following alternative description of nef $b$-divisors:
\begin{lem}\label{lem:mov} Let $X\to S$ be a projective morphism. Then we have 
$$
\Nef(\X/S)=\projlim_\p\overline{\mathrm{Mov}}(X_\p/S)
$$
where the limit is taken over all smooth (or $\Q$-factorial) models $X_\p$. In other words an $\R$-Weil $b$-divisor $W$ is $S$-nef iff $W_\p$ is $S$-movable on each smooth (or $\Q$-factorial) model $X_\p$.  In particular the restriction of (the class of) $W_\p$ to any prime divisor of $X_\p$ is $S$-pseudoeffective. 
\end{lem}

\begin{proof} Let $\a\in N_{n-1}(\X/S)$. Since the latter is endowed with the inverse limit topology the sets
$$
V_{\p,U}:=\{\b\in N_{n-1}(\X/S),\,\b_\p\in U\}
$$
where $\p$ ranges over all smooth models of $X$ and $U\subset N^1(X_\p/S)$ ranges over all conical open neighborhoods of $\a_\p$ form a neighborhood basis of $\a$. 

We infer by definition that $\a$ is $S$-nef iff for every $\p$ and $U$ there exists an  $S$-nef class $\b\in N^1(\X/S)$ such that $\b_\p\in U$. On the other hand since $U$ is conical it is immediate to see that $\b$ may be assumed to be the class of an $S$-globally generated Cartier $b$-divisor, and the result follows. 
\end{proof}

The next result is a limiting case of Lemma \ref{lem:difference}. 
\begin{lem}\label{lem:env} Let $\fra_\bullet$ be a graded linearly bounded denominators. Then the $\R$-Weil $b$-divisor $Z(\fra_\bullet)$ is $X$-nef.
\end{lem}
\begin{proof} Since $\fra_\bullet$ has linearly bounded denominators it is in particular clear that there exists a finite dimensional vector space $V$ of $\R$-Weil divisors on $X$ such that $Z(\fra_m)\in V$ for all $m$. By Lemma \ref{lem:repres} it thus follows that $[\tfrac 1mZ(\fra_m)]$ converges to $[Z(\fra_\bullet)]$ in $N_{n-1}(\X/X)$.  But each $Z(\fra_m)$ is $X$-globally generated by Lemma \ref{lem:difference}, and we thus conclude that $Z(\fra_\bullet)$ is $X$-nef
\end{proof}

\begin{prop}[Negativity Lemma]\label{prop:neg} Let $W$ be an $X$-nef $\R$-Weil $b$-divisor over $X$. Then for each $\p$ we have $W\le\Env_\p(W_\p)$.
\end{prop}
The following argument provides in particular an alternative proof of the well-known \emph{negativity lemma} \cite[Lemma 3.39]{KM}. 

\begin{proof} Let $X_\p$ be a fixed model of $X$. 

\smallskip
{\bf Step 1}. Let $C$ be an $X$-globally generated Cartier $b$-divisor, determined on some model $X_\tau$ that may be assumed to dominate $X_\p$. As in the proof of Lemma \ref{lem:difference} we have $C=Z(\O_X(C))$ since $C$ is $X$-globally generated, and we infer that $C\le\Env_\p(C_\p)$. Indeed $\tau\ge\p$ implies
$$
\O_X(C)=\tau_*\O_{X_\tau}(C_\tau)\subset\p_*\O_{X_\p}(C_\p),
$$
hence 
$$
C=Z(\O_X(C))\le Z\left(\p_*\O_{X_\p}(C_\p)\right)\le\Env_\p(C_\p)
$$
by Proposition \ref{prop:graded-sequence}.

\smallskip
{\bf Step 2}. Let $C$ be an $X$-nef $\R$-Cartier $b$-divisor, determined on a model $X_\tau$ that may again be assumed to be projective over $X$ and to dominate $X_\p$. The class of $C_\tau$ in $N^1(X_\tau/X)$ is $X$-nef, hence belongs to the closed convex cone spanned by the classes of $X$-very ample divisors of $X_\tau$. As in (ii) of Lemma~\ref{lem:repres}, we may then find a sequence of $X$-very ample Cartier divisors $A_j$ on $X_\tau$ and a sequence $t_j\in\R_+^*$ such that $t_j A_j\to C_\tau$ coefficient-wise, while staying in a fixed finite dimensional vector space of $\R$-divisors on $X_\tau$. By Step 1 and Proposition \ref{prop:concave} we have $t_j \overline{A_j}\le\Env_\p(t_j(\overline{A_j})_\p)$ for each $j$. By Corollary \ref{cor:cont} we infer 
$$
\nu(C)=\lim_j t_j\nu(\overline{A_j)})\le\nu\left(\Env_\p(t_j\overline{A_j})\right)=\nu(\Env_\p(C_\p))
$$ 
for each divisorial valuation $\nu$, hence $C\le\Env_\p(C_\p)$. This step recovers in particular the usual statement of the negativity lemma. 

{\bf Step 3}. Let $W$ be an arbitrary $X$-nef $\R$-Weil $b$-divisor. By Lemma \ref{lem:repres} there exists a net $W_j$ of $X$-nef $\R$-Cartier divisors such that $W_j\to W$ coefficient-wise and $W_{j,X}$ stays in a fixed finite dimensional space of $\R$-Weil divisors on $X$. The result now follows by another application of Corollary \ref{cor:cont}. 
\end{proof}
As a consequence we get the following variational characterization of nef envelopes.

\begin{cor}\label{cor:env} If $D$ is an $\R$-Weil divisor on $X_\p$ then $\Env_\p(D)$ is the largest $X$-nef $\R$-Weil $b$-divisor $W$ such that $W_\p\le D$. In particular we have:
\begin{itemize}
\item $\Env_\p(D)=\overline D$ if $D$ is $\R$-Cartier and $X$-nef.
\item The $b$-divisor $\Env_\p(D)$ is $\R$-Cartier, determined by a given $\tau\ge\p$, iff the trace of $\Env_\p(D)$ on $X_\tau$ is $\R$-Cartier and $X$-nef. 
\end{itemize}
\end{cor}

\begin{proof} The $\R$-Weil $b$-divisor $\Env_\p(D)$ is $X$-nef by Lemma \ref{lem:env}. We also clearly have $\tfrac 1 m Z(\p_*\O_{X_\p}(mD))_\p\le D$, hence $\Env(D)_\p\le D$ in the limit. Conversely if $Z$ is an $X$-nef $\R$-Weil $b$-divisor such that $Z_\p\le D$ then $Z\le\Env_\p(Z_\p)\le\Env_\p(D)$ by the negativity lemma.
\end{proof}

As an illustration we now prove:
\begin{prop} Assume that $X$ has klt singularities in the sense that there exists an effective $\Q$-Weil divisor $\D$ such that $K_X+\D$ is $\Q$-Cartier and $(X,\D)$ is klt (cf.~\cite{dFH}). Then $\Env_X(D)$ is an $\R$-Cartier $b$-divisor for every $\R$-Weil divisor $D$ on $X$. When $D$ has $\Q$-coefficients we even have $\Env_X(D)=\tfrac 1m Z(\O_X(mD))$ for some $m$. 
\end{prop}
The result easily follows from \cite[Exercise 109]{kol-ex}, but we provide some details for the convenience of the reader.

Note that the analogous result for $\Env_\p(D)$, $D$ being a Weil divisor on a higher model $X_\p$, fails even when $X$ is smooth (cf. \cite{Cut,Kuro} for an explicit example). 

\begin{proof} Since $(X,\D)$ is klt it follows from \cite[Corollary~1.4.3]{BCHM} that there exists a \emph{$\Q$-factorialization} $\p\colon X_\p\to X$, i.e.~a small birational morphism $\p$ such that $X_\p$ is $\Q$-factorial. Denote by $\hat\D_\p$ and $\hat D_\p$ the strict transforms on $X_\p$ of $\D$ and $D$ respectively. Since $\p$ is small we have $\p^*(K_X+\D)=K_{X_\p}+\hat\D_\p$, which shows that $(X_\p,\hat\D_\p)$ is klt, hence so is $(X_\p,\hat\D_\p+\ep\hat D_\p)$ for $0<\ep\ll 1$. By applying \cite[Corollary~1.4.3]{BCHM} to $\ep\hat D_\p$, which is $\p$-numerically equivalent to $K_{X_\p}+\hat\D+\ep\hat D$ as well as $\p$-big (since $\p$ is birational) we infer the existence of a new $\Q$-factorialization $\tau\colon X_\tau\to X$ such that the strict transform $\hat D_\tau$ of $D$ on $X_\tau$ is furthermore $X$-nef. Since $\tau$ is small it is easily seen that $\tau_*\O_{X_\tau}(m\hat D_\tau)=\O_X(mD)$ for all $m$, hence $\Env_\tau(\hat D_\tau)=\Env_X(D)$, and it follows by Corollary \ref{cor:env} that $\Env_X(D)$ is the $\R$-Cartier $b$-divisor determined by $\hat D_\tau$. 

When $D$ has rational coefficients the base-point free theorem shows that $\hat D_\tau$ is $X$-globally generated, so that 
$$
\bigoplus_{m\ge 0}\O_X(mD)=\bigoplus_{m\ge 0}\tau_*\O_{X_\tau}(m\hat D_\tau)
$$ 
is finitely generated over $\O_X$. We thus have $\Env_X(D)=\tfrac 1m Z(\O_X(mD))$ for some $m$. 
\end{proof}

\subsection{Nef envelopes of Weil $b$-divisors}\label{ss:2.3}
The next result is a variant in the relative case of \cite[Proposition 2.13]{BFJ}  and \cite[Theorem D]{KuMa}:
\begin{prop}\label{prop:largest} Let $W$ be an $\R$-Weil $b$-divisor. If the set of $X$-nef $\R$-Weil $b$-divisors $Z$ such that $Z\le W$ is non-empty then it admits a largest element.
\end{prop}
\begin{defi}
We shall say that \emph{the nef envelope of $W$ is well-defined} if the assumption of the proposition holds. We then denote the largest element in question by $\Env_\X(W)$ and call it the \emph{nef envelope} of $W$.
\end{defi}
\begin{proof}[Proof of Proposition~\ref{prop:largest}] 
Every $Z$ as in the proposition satisfies $Z\le\Env_\p(W_\p)$ for all $\p$ by Corollary \ref{cor:env}, which also implies that $\p\mapsto\Env_\p(W_\p)$ is non-increasing, i.e.
$$
\Env_{\p'}(W_{\p'})\le\Env_\p(W_\p)
$$
whenever $\p'\ge\p$. If there exists at least one $Z$ as above then it follows that $\Env_\X(W):=\lim_\p\Env_\p(W_\p)$ is well-defined as a $b$-divisor and satisfies $\Env_\X(W)\ge Z$ for every such $Z$. There remains to show that $\Env_\X(W)$ is $X$-nef and satisfies $\Env_\X(W)\le W$. But the existence of $Z$ guarantees the existence a finite dimensional vector space $V$ of $\R$-Weil divisors on $X$ such that $\Env_\p(W_\p)_X\in V$ for all $\p$. Since $\Env_\p(W_\p)$ converges to $\Env_\X(W)$ coefficient-wise, we conclude as before by Lemma \ref{lem:repres} that $\Env_\X(W)$ is $X$-nef, whereas $\Env_\X(W)\le W$ follows from $\Env_\p(W_\p)_\tau\le W_\tau$ for $\tau\le\p$ by letting $\p\to\infty$. 
\end{proof} 

\begin{rmk}\label{rmk:Env-comparison}
Note that the proof gives:
$$
\Env_\X(W)=\inf_\p \Env_\p(W_\p)~.
$$
If $W$ is an $\R$-Cartier $b$-divisor then we have
$$
\Env_\X(W)=\Env_\p(W_\p)
$$
for each determination $\p$.  
\end{rmk}
\begin{prop}\label{prop:continuity}
Let $(W_i)_{i\in I}$ be a net of   $b$-divisors decreasing to $W$ such that $\Env_\X(W)$ is well-defined. Then $\Env_\X(W_i)$ is well-defined for every $i$ and
the net decreases to $\Env_\X(W)$.
\end{prop}

\begin{proof}
By assumption $\Env_\X(W)$ is well-defined, so that there exists an X-nef $\R$-Weil
$b$-divisor $Z \le W$. Since $W_i \ge W$ for all $i$, the envelopes $\Env_\X(W_i)$ 
are also well-defined, and form a net that
decreases to a $b$-divisor $Z'\ge \Env_\X(W)$. Pick any $\p$. Since $W_{i,\p} \to W_\p$, 
we have  $\Env_\X(W_i) \le \Env_\p(W_{i,\p}) \to \Env_\p(W_\p)$. Letting $i\to\infty$, we get
$Z' \le \Env_\p(W_\p)$. We conclude  using the preceding remark.
\end{proof}

\begin{prop}\label{prop:pull-env}
Suppose $\f\colon X\to Y$ is a finite dominant morphism of normal varieties. 
Let $W$ be any $\R$-Weil $b$-divisor over $Y$ whose nef envelope $\Env_\Y(W)$ is well-defined. Then $\Env_\X(\f^*W)$ is also well-defined and we have
$$
\Env_\X(\f^*W)=\f^* \Env_\Y(W). 
$$
We similarly have
$$
\Env_X(\f^*D)=\f^*\Env_Y(D)
$$
for every $\R$-Weil divisor $D$ on $Y$. 
\end{prop}

\begin{proof} Since $\Env_\Y(W)$ is $Y$-nef, its pull-back $\phi^*\Env_\Y(W)$ is $Y$-nef as well, hence also $X$-nef. Since we have $\phi^*\Env_\Y(W)\le\phi^*W$ this shows that $\Env_\X(\phi^*W)$ is well-defined and satisfies $\phi^*\Env_\Y(W)\le\Env_\X(\phi^*W)$ by Proposition \ref{prop:largest}. 

Conversely, Lemma \ref{lem:neffunc} below shows that $\f_* \Env_\X(\f^* W)$ is $Y$-nef. Since $\f_*\Env_\X(\f^*W)\le\f_*\f^*W=e(\f)W$ by Proposition \ref{prop:push-cartier} it follows that
$$
\f_*\Env_\X(\f^*W)\le e(\f)\Env_\Y(W)=\phi_*\phi^*\Env_\Y(W)
$$
by Proposition \ref{prop:push-cartier} again, and we conclude by applying Lemma \ref{lem:eff-nul} below to 
$Z:= \Env_\X(\f^* W) - \f^* \Env_\Y(W)$.
\end{proof}

\begin{lem}\label{lem:neffunc} Let $\f\colon X\to Y$ be a \emph{finite} dominant morphism between normal varieties and let $W$ be an $X$-nef $\R$-Weil $b$-divisor over $X$. Then $\phi_*W$ is $Y$-nef. 
\end{lem}

\begin{proof}  By assumption the class of $W$ in $N_{n-1}(\X/X)$ is $X$-nef, hence can be written as the limit of a net of $X$-nef classes of $N^1(\X/X)$. By (b) of Lemma \ref{lem:repres} there exists a net $W_j$ of $X$-nef $\R$-Cartier $b$-divisors such that $W_j\to W$ coefficient-wise and $W_{j,X}$ stays in a fixed finite dimensional vector of $\R$-Weil divisors on $X$. It follows that the divisors $(\phi_*W_j)_Y$ also stay in a fixed finite dimensional vector space of $\R$-Weil divisors on $Y$. Using the definition of $\phi_*$ on Weil $b$-divisors, it is immediate to see that $\phi_*W_j\to\phi_*W$ coefficient-wise. Using (a) of Lemma \ref{lem:repres} it thus follows that $[\phi_*W_j]\to[\phi_* W]$ in $N_{n-1}(\Y/Y)$, and we are thus reduced to the case where $W$ is $\R$-Cartier. 

Now let $\p$ be a determination of $W$. By Corollary \ref{cor:env} we have in particular $W=\Env_\p(W_\p)$, so that the fractional ideals $\fra_m:=\p_*\O(mW_\p)$ satisfy $W=\lim\tfrac 1m Z(\fra_m)$ coefficient-wise, and it is clear that the $Z(\fra_m)_X$ stay in a fixed finite dimensional vector space by monotonicity. We are now reduced to the case where $W=Z(\fra)$ for some fractional ideal, in which case we have $\phi_*Z(\fra)=Z(N_{X/Y}(\fra))$ by (the proof of) Proposition \ref{prop:pushcar}. We conclude that $\phi_*Z(\fra)$ is $Y$-globally generated, hence in particular $Y$-nef, by Lemma \ref{lem:difference}. 
\end{proof}

\begin{lem}\label{lem:eff-nul}
Let $\f\colon X \to Y$ be a proper, surjective, generically finite morphism.
Suppose $Z \ge0$ is an $\R$-Weil $b$-divisor over $X$. Then $\f_*Z =0$ only if $Z=0$.
\end{lem}

\begin{proof}
Suppose that there is a prime divisor $E$ lying in some model $X'$ over $X$ such that
$\ord_E Z >0$. Since $\f$ is generically finite, we can choose a model
$Y'$ over $Y$ such that $E$ maps to a prime divisor $F$ on $Y'$ via the rational map $\f'\colon  X' \rat Y'$ obtained by lifting $\f$. Then $\ord_F(\f_*Z) \ge \ord_E Z>0$, hence $\f_*Z$ cannot be zero.
\end{proof}

\subsection{The case of surfaces and toric varieties}

\begin{thm}\label{thm:mumford} Let $X$ be a normal surface and let $\p\colon X_\p\to X$ be a smooth (or at least $\Q$-factorial) model. 
\begin{enumerate}
\item[(i)] If $D$ is an $\R$-divisor on $X_\p$ then the $b$-divisor $\Env_\p(D)$ is $\R$-Cartier, determined on $X_\p$, and
$$
D=\Env_\p(D)_\p+\left(D-\Env_\p(D)_\p\right)
$$ 
coincides with the relative Zariski decomposition of $D$ with respect to 
$\p \colon X_\p \to X$.
\item[(ii)] If $D$ is an $\R$-Weil divisor on $X$ then $\Env_X(D)=\overline{\p^*D}$ where $\p^*D$ is the numerical pull-back of $D$ in the sense of Mumford.
\end{enumerate}
\end{thm}
Recall that the {\it numerical pull-back} of $D$ is defined as 
the unique $\R$-divisor $D'$ on $X_\p$ such that $\p_*D' = D$ 
and $D'\cdot E =0$ for all $\p$-exceptional divisors $E$. 

\begin{proof} Let us prove (i). The first assertion follows from Corollary \ref{cor:env}, since each movable class is nef when $\dim X=2$. 

The divisor $P:=\Env_\p(D)_\p$ is an $X$-nef $\R$-divisor on $X_\p$ such that $D\ge P$ and $P\ge Q$ for every $X$-nef divisor $Q$ on $X_\p$ such that $D\ge Q$, by Corollary \ref{cor:env} again. 
Write $N := P- D$. Then we have $P\cdot E=0$ for any  prime divisor $E$
in the support of $N$, since otherwise
$P + s E$ is $X$-nef and $\le D$ for $s \ll 1$.
This is one the characterizations of the (relative) Zariski decomposition, see~\cite[p.~408]{Sak84}.
This concludes the proof of (i).

Let us now prove (ii). Let $\p^*D$ be the numerical pull-back of $D$ to $X_\p$. Since $\p^*D$ is $\p$-nef it follows that $C:=\overline{\p^*D}$ is $X$-nef and satisfies $C_X=D$, hence $C\le\Env_X(D)$ by Corollary \ref{cor:env}. Conversely set $D':=\Env_X(D)_\p$. We claim that $D'=\p^*D$. Taking this for granted for the moment we then get $\Env_X(D)\le C$ by the negativity lemma and the result follows. 

Since we have $\p_*D'=D$ by Proposition \ref{prop:envx}, the claim will follow if we show that $D'\cdot E=0$ for each $\p$-exceptional prime divisor $E$ on $X_\p$. This is a consequence of the variational characterization of $\Env_X(D)$. Indeed note that $D'\cdot E\ge 0$ since $D'$ is $\p$-nef by Lemma \ref{lem:mov}. If we assume by contradiction that $D'\cdot E>0$ then $D'+\ep E$ is still $\p$-nef for $0<\ep\ll 1$ and $C:=\overline{D'+\ep E}$ is then an $X$-nef $b$-divisor with $C_X=D$. It follows that $C\le\Env_X(D)$ by Corollary \ref{cor:env}, hence $D'+\ep E\le D'$, a contradiction. 
\end{proof}

Let us now decribe the case of toric varieties. We refer to~\cite{fulton,oda,CLS} for basics on toric varieties. Let $N$ be a free abelian group of rank $n$, and suppose we are given two rational polyhedral fans 
$\Delta, \Delta'$ in $N$ such that $\Delta \subset \Delta'$. For the sake of simplicity we assume 
$\Delta$ and $\Delta'$ have the same support $S$.
Denote by $X(\Delta)$ and $X(\Delta')$ the corresponding toric varieties. Since $\Delta$ is a 
subset of $\Delta'$, we have an induced birational map $\pi\colon  X(\Delta') \to X(\Delta)$.

Let $D$ be an $\R$-Weil toric divisor on $X(\Delta)$. It is given by a real valued function $h_D$ on the set of 
primitive vectors $\Delta(1)$ generating the $1$-dimensional faces of $\Delta$, and $D$ is 
$\R$-Cartier iff $h_D$ extends to a continuous function on $S$ that is linear on each face. In that case $D$ is $\p$-nef iff $h_D$ is convex on the union $S_0$ of all faces of $\Delta'$ that contain a ray in 
$\Delta'(1)\setminus \Delta(1)$. By Corollary \ref{cor:env} it follows that the function attached to $\Env_\p(D)_\p$ is the supremum of all $1$-homogeneous functions on the convex set $S$ such that 
$g \le h_D$ on $\Delta(1)$ and $g$ is convex on the subset $S_0$. 

\begin{eg}\label{eg:counter}
Take $\Delta$ in $\R^3$ the  fan having a single $3$-dimensional cone generated by the four rays 
$(1,0,0), (0,1,0), (0,0,1), (1,1,-1)$. Then $X(\Delta)$ is an affine variety having an isolated 
singularity at the origin and is locally isomorphic to a quadratic cone there.

Let $\Delta'$ be the regular fan having $(1,0,0), (0,1,0), (0,0,1), (1,1,-1), (1,1,0)$ as vertices.
The natural map $X(\Delta') \to X(\Delta)$ is a proper birational map which gives a (non-minimal) 
desingularization of $X(\Delta)$.  Denote by $E_v$ the divisor associated to the corresponding ray 
$v\in\R^3$
either in $X(\Delta)$ or $X(\Delta')$.

Now take $D_1 = E_{100} + E_{010} + E_{001}$, and $D_2 = E_{100} + E_{001} + E_{11-1}$. 
Then $D_1+ D_2$ is a Cartier divisor on $X(\Delta)$ whose support function is given by 
$2x_1+x_2+2x_3$
in the standard coordinates $(x_1, x_2, x_3) \in \R^3$. Hence $\ord_{E_{110}} \Env_X(D_1+D_2) = 3$.
On the other hand, for any convex function $g$ having value $1$ at $(0,0,1)$ and $0$ at $(1,1,-1)$, 
we have $g(1,1,0) \le 1$, hence $\ord_{E_{110}}\Env_X(D_1) \le 1$. The same argument shows that 
$\ord_{E_{110}}\Env_X(D_2) \le 1$, hence
$\ord_{E_{110}}\Env_X(D_1) + \ord_{E_{110}}\Env_X(D_2) < \ord_{E_{110}} (\Env_X(D_1+D_2))$.
\end{eg}

\subsection{Defect ideals}\label{sec:defect}

\begin{defi} The \emph{defect ideal} of an $\R$-Weil divisor $D$ on $X$ is defined as 
$$\frd(D):=\O_X(D)\cdot\O_X(-D).$$
\end{defi}
Note that $\frd(D)\subset\O_X(D-D)=\O_X$ is an ideal sheaf. The following proposition summarizes immediate properties of defect ideals.

\begin{prop}\label{prop:defect}
Let $D,D'$ be $\R$-Weil divisors on $X$. Then we have:
\begin{enumerate}
\item[(i)] $\frd(D + C) = \frd(D)$ for every Cartier divisor $C$.
\item[(ii)] 
$$
\frd(D)\cdot\O_X(D+D')\subset\O_X(D)\.\O_X(D')\subset\O_X(D+D').
$$ 
\item[(iii)] 
$$
\phi^{-1}\frd_X(D)\. \O_Y(\f^*D) \subset\phi^{-1}\O_X(D)\.\O_Y\subset\O_Y(\phi^*D)
$$ 
for every finite dominant morphism $\f\colon Y\to X$.
\item[(iv)]
The sequence 
$$
\frd_\bullet(D) :=(\frd(mD))_{m\ge 0}
$$
is a graded sequence of ideals, and
$$
Z(\frd_\bullet(D))=\Env_X(D)+\Env_X(-D).
$$
\end{enumerate}
\end{prop}

\begin{defi} We shall say that an $\R$-Weil divisor $D$ on $X$ is \emph{numerically Cartier} if $\Env_X(-D)=-\Env_X(D)$. In the special case where $D=K_X$ we shall say that $X$ is \emph{numerically Gorenstein} if $K_X$ is numerically Cartier. 
\end{defi}

\begin{rmk}
If $D$ is a $\Q$-Weil divisor, then 
the property of being numerically Cartier can be 
equivalently checked using valuations, so that $D$ is numerically Cartier
iff given a positive integer $k$ such that $kD$ is an integral divisor, 
for every divisorial valuation $\n$ the sequence $\n(\O_X(mkD)) - \n(\O_X(-mkD))$
is in $o(m)$. 
\end{rmk}

By Proposition \ref{prop:concave} it is straightforward to see that numerically Cartier divisors form an $\R$-vector space. We also have:

\begin{lem}\label{lem:additive} Let $D$ be an $\R$-Weil divisor on $X$. Then $D$ is numerically Cartier iff
$$
\Env_X(D+D')=\Env_X(D)+\Env_X(D')$$ 
for every $\R$-Weil divisor $D'$ on $X$. 
\end{lem}
\begin{proof} Assume that $D$ is numerically Cartier, so that  $\Env_X(-D)=-\Env_X(D)$. Then we have on the one hand $\Env_X(D+D')\ge\Env_X(D)+\Env_X(D')$ and on the other hand $\Env_X(-D)+\Env_X(D+D')\le\Env_X(D')$, and additivity follows. The converse is equally easy and left to the reader.
\end{proof}

\begin{eg}[Surfaces] Since Mumford's pull-back of Weil divisors on surfaces is linear, it follows from Theorem~\ref{thm:mumford} that all $\R$-Weil divisors on a normal surface $X$ are numerically Cartier. 
\end{eg}

\begin{eg}[Toric varieties] If $D$ is a toric $\R$-Weil divisor on a toric variety $X$ then it follows from the discussion from the last section that $D$ is numerically Cartier iff $D$ is already $\R$-Cartier. 
\end{eg}

\begin{eg}[Cone singularities]\label{eg:cone} Let $(V,L)$ be a smooth projective variety endowed with an ample line bundle $L$. Recall that the affine cone over $(V,L)$ is the algebraic variety defined by 
$$
X=C(V,L):=\Spec\left(\bigoplus_{m\ge 0}H^0(V,mL)\right).
$$
If $L$ is sufficiently positive, then $X$ has an isolated normal singularity at its vertex $0\in X $, and is obtained by blowing-down the zero section $E\simeq V$ in the total space $Y$ of the dual bundle $L^*$. We denote by $\p\colon Y\to X$ the contraction map, which is isomorphic to the blow-up of $X$ at $0$. Every divisor $D$ on $V$ induces a Weil divisor $C(D)$ on $X$, and the map $D\mapsto C(D)$ induces an isomorphism $\Pic(V)/\Z\, L\simeq\Cl(X)$ onto the divisor class group of $X$.

\begin{lem}\label{lem:simple} Let $(V,L)$ be a smooth polarized variety and let $D$ be an $\R$-Weil divisor on $V$. Assume that $L$ is sufficiently positive so that $C(V,L)$ is normal.
\begin{itemize}
\item[(1)] $C(D)$ is $\R$-Cartier iff $D$ and $L$ are $\R$-linearly proportional in $\Pic(X)\otimes\R$. 
\item[(2)] $C(D)$ is numerically Cartier iff $D$ and $L$ are numerically proportional in $N^1(V)$. 
\end{itemize}
\end{lem}
\begin{proof} (1) follows from the description of the divisor class group of $X=C(V)$ recalled above. Let us prove (2). Let $\p\colon Y\to X$ be the blow-up of $X$ at its vertex $0$. The restriction to $E\simeq V$ of the strict transform $C(D)'$ is linearly equivalent to $D$. If $D$ is numerically Cartier then the restriction to $E$ of $\Env_X(-C(D))_Y=-\Env_X(C(D))_Y$ is both pseudoeffective and anti-pseudoeffective by Lemma \ref{lem:mov}, so $\Env_X(C(D))_Y$ is numerically equivalent to $0$ in $N^1(Y/X)$. But $\Env_X(C(D))_Y-C(D)'$ is $\p$-exceptional, hence proportional to $E$, and we conclude as desired that $D\equiv C(D)'|_E$ is proportional to $L\equiv -E|E$ in $N^1(V)$. 

Conversely assume that $D\equiv a L$ are proportional in $N^1(V)$. Then $C(D)'$ and $E$ are proportional in $N^1(Y/X)$, hence there exists $t\in\R$ such that $\Env_X(C(D))_Y\equiv -t E$ in $N^1(Y/X)$. Since $-E$ is $X$-ample and the numerical class of $\Env_X(C(D))_Y$ is in the $X$-movable cone it follows that $t\ge 0$, which implies that $\Env_X(C(D))_Y$ is $X$-nef. This in turn shows as in the proof of Theorem \ref{thm:mumford} that the $b$-divisor $\Env_X(C(D))$ is $\R$-Cartier, determined on $Y$ by $C(D)'-aE$. If we replace $D$ by $-D$ then we get that $\Env_X(C(D))$ is determined on $Y$ by $C(-D)'+aE=-\left(C(D)'-aE\right)$, i.e.~$\Env_X(-C(D))=-\Env_X(C(D))$ holds as desired. 
\end{proof}

\end{eg}

We now give a more precise description of defect ideals, which is basically an elaboration of \cite[Theorem 5.4]{dFH}. As a matter of terminology we introduce:
\begin{defi} We say that a determination $\p$ of an $\R$-Cartier $b$-divisor $C$ is a \emph{log-resolution} of $C$ if $X_\p$ is smooth, the exceptional locus $\Exc(\p)$ has codimension one and $\Exc(\p)+C_\p$ has SNC support. 

Another $\R$-Cartier $b$-divisor $C'$ is then said to be \emph{transverse} to $\p$ and $C$ if $\p$ is also a log-resolution of $C+C'$ and $C'_\p$ has no common component with $\Exc(\p)+C_\p$. 
\end{defi}
Every $\R$-Cartier $b$-divisor admits a log-resolution by Hironaka's theorem.

\begin{prop} \label{prop:defect2} Let $D$ be a Weil divisor on $X$ and assume that $X$ is quasi-projective. Then we have
$$
\frd(D) = \sum_E\O_X(-E)
$$
where the sum is taken over the set of all prime divisors $E$ of $X$ such that $D-E$ is Cartier (and this set is in particular non-empty). 

Given a Cartier $b$-divisor $C$ and a joint log-resolution $\p$ of $C$ and $\O_X(D)$ the sum can be further restricted to those $E$ such that $Z(\O_X(E))$ is transverse to $\p$ and $C$.
\end{prop}
\begin{proof} Observe first that 
$$
\O_X(-E)\subset\O_X(-E)\cdot\O_X(E)=\frd(E)=\frd(D)
$$ 
for all effective Weil divisors $E$ such that $D-E$ is Cartier. 

Since $X$ is quasi-projective there exists a line bundle $L$ on $X$ such that $L\otimes\O_X(D)$ is generated by a finite dimensional vector space of global sections $V$, which we view as rational sections of $L$. For each $s\in V$ set $E_s:=D+\dv(s)$, which is an effective Weil divisor congruent to $D$ modulo Cartier divisors. 

We claim that there exists a (non-empty) Zariski open subset $U$ of $V$ such that 
\begin{equation}\label{equ:frd}
\frd(D)=\sum_{s\in U}\O_X(-E_s)
\end{equation}
and
\begin{itemize}
\item $E_s$ is a prime divisor on $X$,
\item $Z(\O_X(E_s))$ is transverse to $\p$ and $C$,
\end{itemize}
for each $s\in U$, which will conclude the proof of Proposition \ref{prop:defect2}. 

Since $\p$ dominates the blow-up of $\O_X(D)$ it is easily seen that the effective divisors 
$$
M_s:=Z(\O_X(E_s))_\p=Z(\O_X(D))_\p+\p^*\dv(s)
$$
move in a base-point free linear system on $X_\p$ as $s$ moves in $V$. We may thus find a non-empty Zariski open subset $U$ of $V$ such that for each $s\in U$ we have

\begin{itemize}
\item $M_s$ has no common component with $\Exc(\p)+C_\p$,
\item $M_s$ is smooth and irreducible,
\item $M_s+\Exc(\p)+C_\p$ has SNC support,
\end{itemize}
where the last two points follow from Bertini's theorem. Since $\p_*M_s=Z(\O_X(D))_X+\dv(s)=E_s$ by Proposition \ref{prop:envx}, we see in particular that $E_s$ is a prime divisor for each $s\in U$ and $Z(\O_X(E_s))$ is transverse to $\p$ and $C$. There remains to show (\ref{equ:frd}). Observe that 
$$s\cdot\O_X(-D)\subset L\otimes\O_X(-\dv(s))\cdot\O_X(-D)=L\otimes\O_X(-E_s)
$$
for each $s\in V$. Since $V$ generates $L\otimes\O_X(D)$ and $U$ is open in $V$ we obtain
$$
L\otimes\frd(D)=L\otimes\O_X(D)\cdot\O_X(-D)
$$
$$
=\sum_{s\in U} s\cdot\O_X(-D)\subset L\otimes\sum_{s\in U}\O_X(-E_s)
$$
and the result follows since $L$ is invertible. 
\end{proof}

%
%%%%%%%%%%%%%%%%%%%%%%
%

\section{Multiplier ideals and approximation}\label{sec:mult-ideal}
In this section $X$ still denotes a normal variety. Our main goal here is to show how to obtain from Takagi's subadditivity theorem for multiplier ideals of pairs a similar statement for the general multiplier ideals defined in \cite{dFH}. This result will in turn enable us to approximate nef envelopes of Cartier divisors from above by nef Cartier divisors, in the spirit of \cite{BFJ}. 

\subsection{Log-discrepancies}
We shall say that an $\R$-Weil divisor $\D$ on $X$ is an \emph{$\R$-boundary} (resp. a $\Q$-boundary, resp. an $m$-boundary) if $K_X+\D$ is $\R$-Cartier (resp.~$K_X+\D$ is $\Q$-Cartier, resp.~$m(K_X+\D)$ is Cartier). 

Let $\omega$ be a rational top-degree form on $X$ and consider the associated canonical $b$-divisor $K_\X$. Given an $\R$-boundary $\D$ on $X$ we define the \emph{relative canonical $b$-divisor} of $(X,\D)$ by
$$K_{\X/(X,\D)}=K_\X-\overline{K_X+\D},$$
which is independent of the choice of $\omega$. If $E$ is a prime divisor above $X$ then $\ord_E K_{\X/(X,\D)}$ is nothing but the \emph{discrepancy} of the pair $(X,\D)$ along $E$. Following \cite{dFH} we introduce on the other hand:

\begin{defi} The \emph{$m$-limiting relative canonical $b$-divisor} is defined by
$$
K_{m,\X/X}:=K_\X+\tfrac 1m Z(\O_X(-mK_X))
$$
and the \emph{relative canonical $b$-divisor} is
$$
K_{\X/X}=K_\X+\Env_X(-K_X).
$$
\end{defi}

They are both independent of the choice of $\omega$ and are exceptional over $X$ by Proposition \ref{prop:envx}. Note that 
$K_{m,\X/X}\to K_{\X/X}$ coefficient-wise as $m\to\infty$. 

Recall that the \emph{log-discrepancy} of a pair $(X,\D)$ along a prime divisor $E$ above $X$ is defined by adding $1$ to the discrepancy. Let us reformulate this by introducing the 'pseudo $b$-divisor' $1_\X$, i.e.~the homogeneous function on the set of divisorial valuations of $X$ such that 
$$
(t\ord_E)(1_\X)=t
$$
for each divisorial valuation $t\ord_E$, so that $\ord_E(K_{\X/(X,\D)}+1_\X)$ is now equal to the log-discrepancy of $(X,\D)$ along $E$. 
We also consider the reduced exceptional $b$-divisor $1_{\X/X}$, 
which takes value 1 on the prime divisors that are exceptional over $X$, and value zero on the prime divisors contained in $X$. 

The following well-known properties show that $K_\X+1_\X$ is better behaved than $K_\X$. 
\begin{lem}\label{lem:increase} Assume that $X$ is smooth and let $E$ be a reduced SNC divisor on $X$. Then we have $K_\X+1_\X\ge\overline{K_X+E}$. 
\end{lem}
This result is \cite[Lemma 3.11]{Kol}, whose proof we reproduce for the convenience of the reader.
\begin{proof} Let $F$ be a smooth irreducible divisor in some model $\p\colon X_\p\to X$. 
We may choose local coordinates $(x_1,...,x_n)$ near the generic point of $\p(F)$ such that the local equation of $E$ writes $x_1\dots x_p=0$ for some $p=0,...,n$, and we let $z$ be a local equation of $F$ at its generic point. We then have $\p^*x_i=z^{b_i}u_i$ where $u_i$ is a unit at the generic point of $F$ and $b_i\in\N$ vanishes for $i>p$.  It follows that $\p^*dx_i=b_i z^{b_i-1} u_i dz+z^{b_i}du_i$, hence 
$$
\ord_F(K_\X-\p^*K_X)=\ord_F(K_{X_\p/X})
$$
$$
=\ord_F(\p^*(dx_1\wedge...\wedge dx_n))\ge -1+\sum_i b_i=-1+\ord_F\overline E.
$$
\end{proof}

\begin{lem}\label{lem:jac} Let $\phi\colon X\to Y$ be a generically finite dominant morphism between normal varieties. Let $\omega_Y$ be a rational top-degree form on $Y$, $\omega_X$ be its pull-back to $X$ and $K_\Y$, $K_\X$ be the associated canonical $b$-divisors. Then we have 
$$
K_\X+1_\X=\phi^*(K_\Y+1_\Y).
$$
\end{lem}

\begin{proof} 
Let $F$ be a prime divisor on a smooth model $Y'$ over $Y$, and pick a smooth model $X'$ over $X$ such that $\phi$ lifts to a morphism $\phi'\colon X'\to Y'$. 
The model $X'$ can be constructed by taking a desingularization of the graph of the rational
map $X \rat Y'$. 
Let $E$ be a prime divisor on $X'$ with $\phi'(E)=F$. We then have $\phi_*\ord_E=b\ord_F$ with $b:=\ord_E(\phi'^*F)$. The same computation as above shows that the ramification order of $\phi'$ at the generic point of $E$ is equal to $b-1$, so that we have 
$$\ord_E(K_{X'}-(\phi')^*K_{Y'})=b-1.$$
It follows that
$$\ord_E(K_{X'})=b\ord_F(K_{Y'})+b-1,$$
i.e.
$$\ord_E(K_\X+1_\X)=(b\ord_F)(K_\Y+1_\Y)$$
as was to be shown. 
\end{proof}

\begin{defi} The \emph{$m$-limiting log-discrepancy $b$-divisor} $A_{m,\X/X}$ and the \emph{log-discrepancy $b$-divisor} $A_{\X/X}$ are the Weil $b$-divisors defined by
$$A_{m,\X/X}:=K_{m,\X/X}+1_{\X/X}$$
and
$$A_{\X/X}:=K_{\X/X}+1_{\X/X}.$$
\end{defi}
Note that $\lim_{m\to\infty}A_{m,\X/X}=A_{\X/X}$ coefficient-wise.  

If $\phi\colon X\to Y$ is a finite dominant morphism recall that the \emph{ramification divisor} $R_\phi$ is the effective Weil divisor on $X$ such that 
$$
K_X=\phi^*K_Y+R_\phi,
$$
where $K_Y$ and $K_X$ are defined by $\omega_Y$ and $\phi^*\omega_Y$ respectively, the divisor $R_\phi$ being again independent of the choice of $\omega_Y$. 
 
\begin{cor}\label{cor:jacrel} Let $\phi:X\to Y$ be a finite dominant morphism between normal varieties. Then we have 
$$
0\le\Env_X(R_\phi)\le\phi^*A_{\Y/Y}-A_{\X/X}\le -\Env_X(-R_\phi)
$$
and the second (resp. third) inequality is an equality when $X$ (resp.~$Y$) is numerically Gorenstein. 
\end{cor}
\begin{proof} Since $\f$ is finite, we have
$$
\phi^*A_{\Y/Y}-A_{\X/X} = \phi^*(K_{\Y/Y}+1_\Y)-(K_{\X/X}+1_\X)
$$
$$=\phi^*\Env_Y(-K_Y)-\Env_X(-K_X)=\Env_X(-\phi^*K_Y)-\Env_X(-K_X)
$$
by Lemma \ref{lem:jac} and Proposition \ref{prop:pull-env}. Now we have on the one hand 
$$
\Env_X(-\phi^*K_Y)=\Env_X(-K_X+R_\phi)\ge\Env_X(-K_X)+\Env_X(R_\phi)
$$
and this is an equality when $X$ is numerically Gorenstein by Lemma \ref{lem:additive}. On the other hand
$$
\Env_X(-K_X)=\Env_X(-\phi^*K_Y-R_\phi)\ge\Env_X(-\phi^*K_Y)+\Env_X(-R_\phi)$$
which is an equality if $Y$ is numerically Gorenstein by Proposition \ref{prop:pull-env} and Lemma \ref{lem:additive}. The result follows, noting that $\Env(R_\phi)\ge 0$ since $R_\phi\ge 0$. 
\end{proof}

\subsection{Multiplier ideals} 
The following definition is a straightforward extension of the usual notion of multiplier ideal with respect to a pair. 

\begin{defi} Let $\D$ be an effective $\R$-boundary on $X$ and let $C$ be an $\R$-Cartier $b$-divisor. We define the \emph{multiplier ideal sheaf} of $C$ with respect to $(X,\D)$ as the fractional ideal sheaf
$$
\J((X,\D);C):=\O_X\left(\lru K_{\X/(X,\D)}+C\rru\right).
$$
\end{defi}

We have in particular
$$
\J((X,\D);C)\subset\O_X(\lru C_X-\D_X\rru),
$$ 
which shows that the (fractional) multiplier ideal is an actual ideal as soon as $C_X\le 0$. By Lemma \ref{lem:increase} we have 
$$
\J((X,\D);C)=\p_*\O_{X_\p}\left(\lru K_{X_\p}-\p^*(K_X+\D)+C_\p)\rru\right)
$$
for each joint log-resolution $\p$ of $(X,\D)$ and $C$. This shows in particular that $\J((X,\D);C)$ is coherent, and in case $C=Z(\fra^c)$ for a coherent ideal sheaf $\fra$ and $c>0$ we recover
$$
\J((X,\D);Z(\fra^c))=\J((X,\D);\fra^c)
$$
where the right-hand side is defined in \cite[Definition 9.3.56]{Laz1}. 

We similarly introduce the following straightforward generalization of the notion of multiplier ideal defined in \cite{dFH}:

\begin{defi} Let $C$ be an $\R$-Cartier $b$-divisor over $X$. 
\begin{itemize}
\item For each positive integer $m$ the \emph{$m$-limiting multiplier ideal sheaf} of $C$ is the fractional ideal sheaf
$$
\J_m(C) :=\O_X\left(\lru K_{m,\X/X}+ C\rru\right).
$$
\item The \emph{multiplier ideal sheaf}  $\J(C)$ is the unique maximal element in the family of fractional ideal sheaves $\J_m(C)$, $m\ge 1$. 
\end{itemize}
\end{defi}

Here again Lemma \ref{lem:increase} implies that 
 $$
\J_m(C) =\p_*\O_{X_\p}\left(\lru K_{X_\p} + \tfrac1m \,Z (\O_X(-mK_X))_\p+ C_\p\rru\right)
$$
for each joint log-resolution $\p$ of $\O_X(-mK_X)$ and $C$, which shows in particular that $\J_m(C)$ is coherent. We also have 
$$
\J_m(C)\subset\O_X(\lru C_X\rru),
$$ 
which implies the existence of a unique maximal element in the set of fractional ideals $\{\J_m(C),\,m\ge 1\}$, by using as usual
$$
\tfrac{1}{lm} \,Z (\O_X(-lmK_X))\ge\max\left(\tfrac1m \,Z (\O_X(-mK_X)),\tfrac1l \,Z (\O_X(-lK_X))\right).
$$

As in \cite{dFH} we now relate the above notions of multiplier ideals, obtaining in particular a more precise version of~\cite[Theorem 5.4]{dFH}.

\begin{thm}\label{thm:compat} Assume that $X$ is quasi-projective, let $C$ be an $\R$-Cartier $b$-divisor and let $m\ge 2$. Then we have 
$$\frd(mK_X)=\sum_\D\O_X(-m\D)$$
where $\D$ ranges over the set of all effective $m$-boundaries such that 
$$
\J_m(C)=\J((X,\D);C),
$$
(so that this set is in particular non-empty).
\end{thm}
\begin{proof} Let $\p$ be a joint log-resolution of $\fra$ and $\O_X(-mK_X)$. By Proposition \ref{prop:defect2} applied to $-mK_X$ we have 
$$
\frd(mK_X)=\sum_E\O_X(-E)
$$
where $E$ ranges over all prime divisors such that $mK_X+E$ is Cartier and $Z(\O_X(E))$ is transverse to $\p$ and $C$. There remains to set $\D:=\tfrac 1m E$ and to observe that $\lfloor\D\rfloor=0$, so that $\J_m(C)=\J((X,\D);C)$ by Lemma \ref{lem:compat} below.
\end{proof}

\begin{lem}\label{lem:compat} Let $C$ be an $\R$-Cartier $b$-divisor, let $\p$ be a joint log-resolution of $C$ and $\O_X(-mK_X)$ and let $\D$ be an effective $m$-boundary. 
\begin{itemize}
\item We have 
$$
\J((X,\D);C)\subset\J_m(C).
$$
\item If $\lfloor\D\rfloor=0$ and $Z(\O_X(m\D))$ is transverse to $\p$ and $C$ then
$$
\J((X,\D);C)=\J_m(C).
$$
\end{itemize}
\end{lem}
\begin{proof} Since $m(K_X+\D)$ is Cartier we have 
$$
\O_X(-mK_X))=\O_X(m\D)\cdot\O_X(-m(K_X+\D))
$$
hence
\begin{equation}\label{equ:mbound}
\tfrac 1m Z(\O_X(-mK_X))=\tfrac 1m Z(\O_X(m\D))-\overline{K_X+\D}
\end{equation}
and the first point follows because $Z(\O_X(m\D))\ge 0$. 

Assume now that $\lfloor\D\rfloor=0$ and that $Z(\O_X(m\D))$ is transverse to $\p$ and $C$. By (\ref{equ:mbound}) we have 
$$
\lru K_{X_\p}-\p^*(K_X+\D)+ C_\p\rru=\lru K_{X_\p}+\tfrac1m \,Z (\O_X(-mK_X))_\p+ C_\p\rru-\lfloor \tfrac 1m Z(\O_X(m\D))_\p\rfloor.
$$
Indeed, by the transversality assumption $\tfrac 1m Z(\O_X(m\D)_\p$ has no common component with $C_\p$ and no common component with $K_{X_\p}+\frac 1m Z(\O_X(-mK_X))_\p$, the latter being $\p$-exceptional by Proposition \ref{prop:envx}. But by transversality we also have $\tfrac 1m Z(\O_X(m\D))_\p=\widehat\D_\p$, the strict transform of $\D$ on $X_\p$, and the result follows since $\lrd\widehat\D_\p\rrd=0$.
\end{proof}

As a consequence we get the following extension of \cite[Corollary 5.5]{dFH} to $b$-divisors. 
\begin{cor}\label{cor:classical}
Let $X$ be a normal quasi-projective variety and let $C$ be an $\R$-Cartier $b$-divisor. 
\begin{itemize}
\item The $m$-limiting multiplier ideal $\J_m(C)$ is the largest element of the set of multiplier ideals $\J((X,\D);C)$ where $\D$ ranges over all effective $m$-boundaries on $X$.
\item The multiplier ideal $\J(C)$ is the largest element of the set of multiplier ideals $\J((X,\D);C)$ where $\D$ ranges over all effective $\Q$-boundaries on $X$.
\end{itemize}
\end{cor}

We will need the following variant of Lemma~\ref{lem:compat}.

\begin{cor}\label{lem:compat-var}
With the same assumption as in Lemma~\ref{lem:compat}, if $m \ge 3$ then  
we can find an effective $m$-compatible 
boundary $\D$ such that 
$$
\J\left((X,\D); C + \tfrac 1m Z(\O_X(-m\D))\right) = 
\J_m\left(C + \tfrac 1m Z(\frd(mK_X))\right).
$$
\end{cor}

\begin{proof}
The problem is local, so we can assume that $X$ is affine.
Let $\p$ be as in the statement of Lemma~\ref{lem:compat}.
If $f \in \frd(mK_X)$ is a general element, then
$\ord_F(f) = \ord_F(\frd(mK_X))$ for every $\p$-exceptional prime divisor $F$. 
By Theorem~\ref{thm:compat} and its proof, we can find an effective $m$-boundary of the form
$\D = \tfrac 1m E$
where $E$ is a prime divisor, such that $f \in \O_X(-m\D) \subset \frd(mK_X)$ and
$$
\J\left((X,\D); C + \tfrac 1m Z(\frd(mK_X))\right) = 
\J_m\left(C + \tfrac 1m Z(\frd(mK_X))\right).
$$
Note that $\ord_F(\O_X(-m\D)) = \ord_F(\frd(mK_X))$ 
for every $\p$-exceptional prime divisor $F$. Thus, bearing in mind that
$Z(\frd(mK_X))$ is exceptional as $X$ is regular in codimenion one, we have
$$
Z(\O_X(-mD))_\p = Z(\frd(mK_X))_\p - m\^\D_\p.
$$
Since $\^\D_\p$ does not share any component with $C_\p$, and
$\lrd 2\^\D_\p \rrd = 0$, we see that
\begin{multline*}
\lru K_{X_\p}-\p^*(K_X+\D)+ C_\p + \tfrac 1m Z(\O_X(-m\D))\rru = \\
=\lru  K_{X_\p}-\p^*(K_X+\D) + C_\p + \tfrac 1m Z(\frd(mK_X))\rru ,
\end{multline*}
which gives
$$
\J\left((X,\D); C + \tfrac 1m Z(\O_X(-m\D))\right) = 
\J\left((X,\D); C + \tfrac 1m Z(\frd(mK_X))\right).
$$
This completes the proof of the corollary.
\end{proof}

Asymptotic multiplier ideals can also be generalized to this setting. 
For short, we say that a sequence of $\R$-Cartier $b$-divisors
$Z_\bullet = (Z_m)_{m \ge 1}$ is a \emph{bounded graded sequence}
if there is a $\R$-Cartier $b$-divisor $B$ such that
$B \ge \tfrac{1}{km}Z_{km}\ge\max\{\tfrac 1kZ_k,\tfrac 1mZ_m)$
for all $m,k \ge 0$. The following definition relies on
the Noetherian property. 

\begin{defi}
Let $\D$ be an effective $\R$-boundary on $X$, let $C$ be an $\R$-Cartier $b$-divisor,
and let $Z_\bullet = (Z_m)_{m \ge 1}$ be a bounded graded sequence of
$\R$-Cartier $b$-divisors. 
\begin{itemize}
\item
The \emph{asymptotic multiplier ideal sheaf}  $\J((X,\D);C + Z_\bullet)$ 
with respect to $(X,\D)$ 
is the unique maximal element in the family of 
multiplier ideal sheaves $\J((X,\D); C + \tfrac 1k Z_k)$, $k\ge 1$. 
\item
The \emph{asymptotic multiplier ideal sheaf} $\J(C + Z_\bullet)$ 
is the unique maximal element in the family of 
multiplier ideal sheaves $\J(C + \tfrac 1k Z_k)$, $k\ge 1$. 
\end{itemize}
\end{defi}

\begin{lem}\label{rmk:asymp-mult-ideal}
$\J(C + Z_\bullet) = \J_m(C + \tfrac 1m Z_m)$ for every sufficiently divisible $m$.
\end{lem}

\begin{proof}
We have $\J(C + Z_\bullet) = \J(C + \tfrac 1p Z_p)$ for every sufficiently divisible $p$. 
If we fix any such $p$, then we have $\J(C + \tfrac 1pZ_p) = \J_m(C + \tfrac 1p Z_p)$
for every sufficiently divisible $m$. In particular, if we pick $m$ to be a multiple of $p$,
then we have
$$
\J(C + Z_\bullet) = \J(C + \tfrac 1p Z_p) = \J_m(C + \tfrac 1p Z_p)
\subset \J_m(C + \tfrac 1m Z_m) \subset \J(C + \tfrac 1m Z_m) \subset \J(C + Z_\bullet). 
$$
The lemma follows.
\end{proof}

In the case $C = cZ(\fra)$ for some $c \ge 0$ and some nonzero ideal sheaf 
$\fra \subseteq\O_X$, and $Z_k = d Z(\frb_k)$ for some $d \ge 0$ and some
graded sequence of ideal sheaves $\frb_\bullet = (\frb_m)_{m \ge 0}$, then 
we also use the notation
$$
\J((X,\D);\fra^c\.\frb_\bullet^d),
\qquad\J_m(\fra^c\.\frb_\bullet^d),
\qquad \J(\fra^c\.\frb_\bullet^d),
$$
to denote $\J((X,\D);C + Z_\bullet)$,
$\J_m(C + Z_\bullet)$, and
$\J(C + Z_\bullet)$, respectively.

\begin{prop}\label{prop:extract-a}
For every nonzero ideal sheaf $\fra \subset \O_X$, we have
$\fra\.\J(\O_X) \subset \J(\fra)$.
\end{prop}

\begin{proof}
Let $f = gh$, with $g \in \fra$ and $h \in \J(\O_X)$.
Then $Z(f) = Z(g) + Z(h) \le Z(g) + K_{m,X_m/X}$
for every $m \ge 1$, which implies the statement.
\end{proof}

\subsection{Subadditivity and approximation}

Recall that the Jacobian ideal sheaf $\Jac_X \subset \O_X$ of $X$ is defined as the
$n$-th Fitting ideal $\Fitt^n(\Om^1_X)$ with $n=\dim X$. 

Takagi obtained in \cite{Tak3} the following general subadditivity result for multiplier ideals with respect to a pair: 

\begin{thm}\label{Tak3}\cite{Tak3} 
Let $X$ be a normal variety and let $\D$ be an effective $\Q$-Weil divisor such that $m(K_X+\D)$ is Cartier for some integer $m > 0$. If $\fra,\frb$ are two nonzero coherent ideal sheaves on $X$ and $c,d\ge 0$ then we have  
$$
\Jac_X \.\, \J\big((X,\D);\fra^c\.\frb^d\.\O_X(-m\D))^{1/m}\big) \subset
\J((X,\D);\fra^c)\.\,\J((X,\D);\frb^d).
$$
\end{thm}

Note that when $X$ is smooth and $\D=0$ the statement
reduces to the original subadditivity theorem of \cite{DEL}. 
Takagi gives two independent proofs of this result. The first one is based on  positive characteristic technics and relies on the corresponding statement for test ideals.
The other one builds on the work of Eisenstein~\cite{eisenstein} and relies on Hironaka's desingularization theorem.

We now show how to deduce from Takagi's result a subadditivity theorem for multiplier ideals in the sense of \cite{dFH}. 

\begin{thm}[Subadditivity]\label{thm:subadditivity}
Let $X$ be a normal variety. If $\fra,\frb$ are two nonzero coherent ideal sheaves on $X$ and $c,d \ge 0$ then we have 
$$
\Jac_X \.\, \J(\fra^c\.\frb^d\.\frd_\bullet(K_X)) \subset \J(\fra^c)\.\,\J(\frb^d). 
$$
\end{thm}

The results in \cite{Tak,Sch}, combined, suggest the possibility that the correction
term $\frd_\bullet(K_X)$ in the left-hand side might be unnecessary.

\begin{proof} 
%The result is local so we may assume that $X$ is affine.
By Lemma~\ref{rmk:asymp-mult-ideal} we have
$$
\J(\fra^c\.\frb^d\.\frd_\bullet(K_X)) = \J_m(\fra^c\.\frb^d\. \frd(mK_X)^{1/m})
$$ 
for every sufficiently divisible $m$. 
Fix any such $m$; we can assume that $m \ge 3$. 
By Corollary~\ref{lem:compat-var}, we can 
find an effective $m$-compatible boundary $\D$ such that
$$
\J_m(\fra^c\cdot\frb^d\.\frd(mK_X)^{1/m}) = 
\J((X,\D);\fra^c\cdot\frb^d \.\O_X(-m\D)^{1/m}).
$$
Now we apply Theorem~\ref{Tak3} to get the inclusion 
$$
\J((X,\D);\fra^c\cdot\frb^d \.\O_X(-m\D)^{1/m}) \subset 
\J((X,\D);\fra^c)\.\,\J((X,\D);\frb^d).
$$
We conclude by observing that $\J((X,\D);\fra^c)\subset\J_m(\fra^c)\subset\J(\fra^c)$, and the similar statement for $\frb^d$ hold at any rate, by Lemma \ref{lem:compat}.
\end{proof}

\begin{thm}\label{thm:approx}
Let $X$ be a normal variety and let $\fra_\bullet$ be a graded sequence of ideal sheaves on $X$. Then we have
$$
Z(\Jac_X) + Z(\frd_\bullet(K_X))\le Z (\J(\fra_\bullet)) -Z(\fra_\bullet)\le A_{\X/X}.
$$
In particular $\tfrac 1k Z(\J(\fra_\bullet^k))\to Z(\fra_\bullet)$ coefficient-wise as $k\to\infty$, uniformly with respect to $\fra_\bullet$. 
\end{thm}

This result is an extension to the singular case of~\cite[Proposition~3.18]{BFJ}, which was in turn a direct elaboration of the main result of~\cite{ELS}. 

\begin{proof}
For each $k\ge 1$ we have 
$$
Z(\J(\fra_k^{1/k}))\le\tfrac 1k Z(\fra_k) + A_{\X/X}
$$
by definition of multiplier ideals, and the right-hand inequality follows. 

Regarding the other inequality, let for short 
$\frd_\bullet = (\frd_m)_{m \ge 0} := \frd_\bullet(K_X)$.
A recursive application of Theorem \ref{thm:subadditivity} yields 
$$
\Jac_X^{k-1}\.\,\J(\fra_k\.\frd_\bullet^{k-1}) \subset \J(\fra_k^{1/k})^k.
$$
On the other hand, by Proposition~\ref{prop:extract-a} and the definition of 
asymptotic multiplier ideal, we have
$$
\fra_k\.\frd_{k-1}\.\,\J(\O_X) \subset
\J(\fra_k\.\frd_{k-1}) \subset
\J(\fra_k\.\frd_\bullet^{k-1})).
$$
In terms of $b$-diviors, this gives
$$
(k-1)Z(\Jac_X) + Z(\fra_k) + Z(\frd_{k-1}) + Z(\J(\O_X))\le k Z (\J(\fra_k^{1/k}).
$$
We conclude by dividing by $k$ and letting $k\to\infty$. 
\end{proof}

%
%%%%%%%%%%%%%%%%%%%%%%
%

\section{Normal isolated singularities}\label{sec:isolated}

From now on $X$ has an \emph{isolated} normal singularity at a given point $0\in X$, and $\frm\subset\O_X$ denotes the maximal ideal of $0$. We first show how to extend to this setting the intersection theory of nef $b$-divisors introduced in the smooth case in \cite{BFJ}. The main ingredient to do so is the approximation theorem from the previous section. We next define the \emph{volume} of $(X,0)$ as the self-intersection of the nef envelope of the log-canonical $b$-divisor. 

\subsection{$b$-divisors over $0$}
Observe that every Weil $b$-divisor $W$ over $X$ decomposes in a unique way as a sum
$$
W = W^0 + W^{X \smallsetminus 0},
$$
where all irreducible components of $W^0$ have center $0$,
and none of $W^{X\smallsetminus 0}$ have center $0$. If $W = W^0$, then we say that $W$ {\it lies over} $0$ and we denote by 
$$
\Div(\X,0) \subset \Div(\X)
$$
the subspace of all Weil $b$-divisors over $0 \in X$. An element of $\Div_\R(\X,0)$ is the same thing as a real-valued homogeneous function on the set of divisorial valuations on $X$ centered at $0$.

\begin{eg} For every coherent ideal sheaf $\fra$ on $X$ we have 
$$
Z(\fra)^0 = \lim_{k \to \infty}Z(\fra + \frm^k).
$$
\end{eg}

On the other hand we say that a Cartier $b$-divisor $C\in\CDiv(\X)$ is \emph{determined} over $0$ if it admits a determination $\p$ which is an \emph{isomorphism away from $0$}, and we say that $C$ is a \emph{Cartier $b$-divisor over $0$} if $C$ furthermore lies over $0$.
We denote by $\CDiv(\X,0)$ the space of Cartier $b$-divisors over $0$.
There is an inclusion 
$$\CDiv(\X,0) \subset \CDiv(\X)\cap\Div(\X,0)$$
but this is in general not an equality. The following example was kindly suggested to us by Fulger.

\begin{eg}
Consider $(X,0) = (\C^3,0)$. 
Let $f\colon Y \to X$ be the morphism given by first taking the blow-up $f_1\colon Y_1 \to X$ along a line $L$ passing through $0$, and then taking the blow-up $f_2 \colon Y \to Y_1$ at a point $p$ on the fiber of $f_1$ over $0$. Let $E$ be the exceptional divisor of $f_1$ and $D$ be the exceptional divisor of $f_2$. Note that $D$ lies over $0$.
We claim that the Cartier $b$-divisor $\ov{D}$ cannot be determined over $0$.
If that were the case, then there would exist a model $X' \to X$ that is an isomorphism outside $0$, and a divisor $D'$ on $X'$ such that $\ov{D} = \ov{D'}$ as $b$-divisors over $X$. 
In order to show that this is impossible, consider two sections of the $\P^1$-bundle 
$E \to L$ induced by $f_1$, the second one passing through $p$ but not the first, 
and let $C_0$ and $C_1$ be their respective proper transforms on $Y$, so that $D\.C_i = i$. If $L'$ is the proper transform of $L$ on $X'$, then projection formula yields $D\. C_i = D'\.L'$, and thus $D\.C_0 = D\.C_1$. This gives a contradiction.  
\end{eg}

\begin{rmk}
The previous example can be understood torically. Consider in general 
$(X,0) = (\C^n,0)$. It is a toric variety defined by the regular fan $\D_0$ in $\R^n$
having the canonical basis as vertices. Any proper birational toric modification
$\pi \colon  X(\D) \to \C^n$ is determined by  a refinement $\D$ of $\D_0$. We assume $X(\D)$ to be smooth. Denote by $V(\sigma)$ the torus invariant subvariety of $X(\D)$
associated to a face $\sigma$ of $\D$.
For any vertex $v$ of $\D$,  let $D(v)$ be the Cartier $b$-divisor determined in $X(\D)$ by the divisor  $V(\R_+ v)$.
Observe that for any face $\sigma$ of $\D$, we have $\pi(V(\sigma)) =0$ iff $\sigma$  is included in  the open cone $(\R^*)^n_+$.
Whence $D(v)$ lies over $0$ iff $v \in (\R^*)^n_+$.
And $D(v)$ is determined over $0$ iff any face of $\D$ containing $v$ is included in 
$(\R^*)^n_+$.
\end{rmk}

\begin{eg}\label{eg:ppal} Let $\fra\subset\O_X$ be an ideal. Then $Z(\fra)$ is determined over $0$ as soon as $\fra$ is locally principal outside $0$ since the normalized blow-up of $X$ along $\fra$ is then an isomorphism away from $0$. If $\fra$ is furthermore $\frm$-primary then 
$Z(\fra)$ is a Cartier $b$-divisor over $0$. 
\end{eg}

\begin{defi}
We shall say than an $\R$-Weil $b$-divisor $W$ over $0$ is \emph{bounded below} if there exists $c>0$ such that $W\ge c Z(\frm)$.  
\end{defi}
Recall that $Z(\frm)\le 0$, so that the condition means that the function 
$\nu\mapsto\nu(W)/\nu(\frm)$ is bounded below on the set of divisorial valuations centered at $0$. 

\begin{prop}\label{prop:bounded} $(A_{\X/X})^0$ is bounded below. 
\end{prop}

\begin{proof} 
Since $Z(\O_X(-K_X)) \le \Env_X(K_X)$ by the definition of nef envelope, 
it follows
that $A_{\X/X}\ge A_{1,\X/X}$, and hence it suffices to check that $(A_{1,\X/X})^0$ is bounded below.
Let $\p$ be a resolution of the singularity of $X$, chosen to be an isomorphism away from $0$. For each divisorial valuation $\nu$ centered at $0$ we have
$$
\nu(A_{1,\X/X})=\nu\left((K_\X+1_\X)-\overline{K_{X_\p}}\right)+
\nu\left(\overline{K_{X_\p}}+Z(\O_X(-K_X))\right).
$$
The first term in the right-hand side is non-negative since it is equal to the log-discrepancy of the smooth variety $X_\p$ along $\nu$. On the other hand the Cartier $b$-divisor $\left(\overline{K_{X_\p}}+Z(\O_X(-K_X))\right)$ is determined over $0$ since $\O_X(-K_X)$ is locally principal outside $0$ by assumption (cf. Example \ref{eg:ppal}) and it also lies over $0$ by Proposition \ref{prop:envx}. We thus see that 
$$
\left(\overline{K_{X_\p}}+Z(\O_X(-K_X))\right)\in\CDiv(\X,0)
$$
and we conclude by Lemma \ref{lem:bounded} below.
\end{proof}

\begin{lem}\label{lem:bounded} Every $C\in\CDiv(\X,0)$ is bounded below.
\end{lem}
\begin{proof} Let $\p$ be a determination of $C$ which is an isomorphism away from $0$. The result follows directly from the fact that $Z(\frm)_\p$ contains every $\p$-exceptional prime divisor $E$ in its support (since $\ord_E$ is centered at $0$).
\end{proof}

\subsection{Nef $b$-divisors over $0$}
We shall that an $\R$-Weil $b$-divisor over $0$ is \emph{nef} if its class in $N^1(\X/X)$ is $X$-nef. If $W$ is an $\R$-Weil $b$-divisor over $0$ that is bounded below then $\Env_\X(W)$ is well-defined, nef, and it lies over $0$. 

By a result of Izumi~\cite{izumi} for every two divisorial valuations $\nu,\nu'$ on $X$ centered at $0$ there is a constant $c = c(\n,\n') >0$ such that 
$$
c^{-1} \nu (f) \le \nu'(f) \le c \nu(f)
$$
for every $f\in \O_X$. This result extends to nef $b$-divisors by approximation:

\begin{thm}\label{thm:izumi} Given two divisorial valuations $\nu,\nu'$ centered at $0$ there exists $c>0$ such that 
$$c\nu(W)\le\nu'(W)\le c^{-1}\nu(W)$$
for every $X$-nef $\R$-Weil $b$-divisor $W$ such that $W\le 0$ (which amounts to $W_X\le 0$ by the negativity lemma). 
\end{thm}
\begin{proof} Since $\Env_\p(W_\p)$ decreases coefficient-wise to $W$ as $\p\to\infty$ by Proposition \ref{prop:largest}, it is enough to treat the case where $W=\Env_\X(C)$ for some $\R$-Cartier $b$-divisor $C\le 0$. But we then have 
$$W=\lim_{m\to\infty}\tfrac 1m Z(\O_X(mC))$$
with $\O_X(mC)\subset\O_X$ so we are reduced to the case of an ideal, for which the result directly follows from Izumi's theorem.
\end{proof}

\begin{cor}\label{cor:izumi}
For each $X$-nef $\R$-Weil $b$-divisor $W$ such that $W\le 0$ and $W^0\neq 0$ there exists $\ep>0$ such that
$$
W\le\ep Z(\frm). 
$$
\end{cor}
\begin{proof} Since $W^0\neq 0$ there exists a divisorial valuation $\nu_0$ centered at $0$ such that $\nu_0(W)<0$, and it follows that $\nu(W)<0$ for \emph{all} divisorial valuations centered at $0$ by Theorem \ref{thm:izumi}.

Now let $\p$ be the normalized blow-up of $\frm$. Since $W_\p$ contains each $\p$-exceptional prime in its support there exists $\ep>0$ such that $W_\p\le\ep Z(\frm)_\p$ and the result follows by the negativity lemma.
\end{proof}

For nef envelopes of Weil divisors with integer coefficients this result can be made uniform as follows:
\begin{thm}\label{thm:uniform}
There exists $\ep>0$ only depending on $X$ such that
$$
\Env_X(-D)\le\ep Z(\frm)
$$
for all effective Weil divisors (with integer coefficients) $D$ on $X$ containing $0$. 
\end{thm}

\begin{proof}
By Hironaka's resolution of singularities we may choose a smooth birational model $X_\p$ which dominates the blow-up of $\frm$ and is isomorphic to $X$ away from $0$, and such that there exists a $\p$-ample and $\p$-exceptional Cartier divisor $A$ on $X_\p$. If we denote by $E_1,...,E_r$ the $\p$-exceptional prime divisors then $A=-\sum_j a_j E_j$ with $a_j\ge 1$ by the negativity lemma. 

By the negativity lemma the desired result means that there exists $\ep>0$ such that 
for each effective Weil divisor $D$ through $0$ on $X$ we have
$$
\Env_X(-D)_\p\le\ep Z(\frm)_\p.
$$
If we set $c_j(D):=-\ord_{E_j}\Env_X(-D)$ then in view of Theorem \ref{thm:izumi} this amounts to proving the existence of $\ep>0$ such that 
$$
\max_{1\le j\le r} c_j(D)\ge\ep 
$$
for each $D$. 
Note that 
\begin{equation}\label{equ:cji}
\sum_jc_j(D) E_j=-\Env_X(-D)_\p-\widehat D_\p
\end{equation}
by Proposition \ref{prop:envx}. Now we have on the one hand 
$$
-A^{n-1}\cdot\Env_X(-D)_\p=\sum a_j E_j\cdot A^{n-2}\cdot\Env_X(-D)_\p
$$
$$
=\sum_j a_j (A|_{E_j})^{n-2}\cdot(\Env_X(-D)_\p|_{E_j})\ge 0
$$
since $A|_{E_j}$ is ample and $\Env_X(-D)_\p|_{E_j}$ is pseudo-effective by Lemma \ref{lem:mov}. On the other hand 
$$
-A^{n-1}\cdot\widehat D_\p=\sum_j a_j (A|_{E_j})^{n-2}\cdot(\widehat D_\p|_{E_j})\ge 1
$$ 
since $\widehat D_\p|_{E_j}$ is an effective Cartier divisor on $E_j$, and is non-zero for at least one $j$. We thus get
$\sum_jc_j(D)(E_j\cdot A^{n-1})\ge 1$ from (\ref{equ:cji}) and we infer that 
$$
\max_j c_j(D)\ge\ep:=1/\max_j(E_j\cdot A^{n-1}).
$$ 
\end{proof}

We conclude this section by the following crucial consequence of Theorem \ref{thm:approx}.
 
\begin{thm}\label{thm:approxisol} Let $C\in\CDiv(\X,0)$ and set $W:=\Env_\X(C)$. Then there exists a sequence of $\frm$-primary ideals $\frb_k$ and a sequence of positive rational numbers $c_k\to 0$ such that:
\begin{itemize}
\item $c_k Z(\frb_k)\ge W$ for all $k$.
\item $\lim_{k\to\infty} c_k Z(\frb_k)=W$ coefficient-wise. 
\end{itemize}
\end{thm}
\begin{proof} Consider the graded sequence of $\frm$-primary ideals $\fra_m:=\O_X(mW)=\O_X(mC)$ and set $\frb_k:=\J(\fra_\bullet^k)$. By Theorem \ref{thm:approx} we have in particular 
$$
Z(\frb_k)\ge k W+Z(\frd(K_X))+Z(\Jac_X)
$$
and $\tfrac 1k Z(\frb_k)\to W$ coefficient-wise. Since $0\in X$ is an isolated singularity we see that both $\frd(K_X)$ and $\Jac_X$ are $\frm$-primary ideals and Lemma \ref{lem:bounded} yields $c>0$ such that
$$Z(\frd(K_X))+Z(\Jac_X)\ge c Z(\frm).$$
On the other hand there exists $\ep>0$ such that $W\le\ep Z(\frm)$ by Corollary \ref{cor:izumi} and we conclude that there exists $c>0$ such that 
$$
Z(\frb_k)\ge k W+c W
$$
for all $k$. There remains to set $c_k:=1/(k+c)$. 
\end{proof}

\subsection{Intersection numbers of nef $b$-divisors}

We indicate in this subsection how to extend to the singular case the local intersection theory of nef $b$-divisors introduced in~\cite[\S 4]{BFJ} in the smooth case. The main point is to replace the approximation result \cite[Proposition 3.13]{BFJ} by Theorem \ref{thm:approxisol}. 

Let $C_1,...,C_n$ be $\R$-Cartier $b$-divisors over $0$. 
Pick a common determination $\p$ which is an isomorphism away from $0$ and set
$$
C_1\cdot ...\cdot C_n:= C_{1,\p}\cdot ...\cdot  C_{n,\p}.
$$ 
The right-hand side is well-defined since $C_{1,\p}$ has compact support and it does not depend on the choice of $\p$ by projection formula, since $C_{i,\p'} = \m^*C_{i,\p}$
for any higher model $\m \colon X_{\p'} \to X_\p$.

The following property is a direct consequence of the definition of $Z(\fra_i)$
and the formula displayed in \cite[p.~92]{Laz1}.

\begin{prop}\label{prop:ram}
Let $\fra_1, \dots, \fra_n\subset \O_X$ be $\frm$-primary ideals. Then
$$
- Z(\fra_1)\cdot ...\cdot Z(\fra_n) = e ( \fra_1, ..., \fra_n)
$$
where $e(\fra_1, ..., \fra_n)$ denotes the mixed multiplicity 
(see e.g.~\cite[p.~91]{Laz1} for a definition).
\end{prop}

The intersection numbers of \emph{nef} $\R$-Cartier $b$-divisors $C_1,...,C_n$, $C'_1,...,C'_n$ over $0$ satisfy the monotonicity property: 
$$
C_1\cdot ...\cdot C_n\le C'_1\cdot ...\cdot  C'_n
$$
if $C_i\le C'_i$ for each $i$.

\begin{defi}\label{defi:inter} If $W_1,...,W_n$ are arbitrary nef $\R$-Weil $b$-divisors over $0$ we set
$$
W_1\cdot ...\cdot W_n:= \inf_{C_i\ge W_i}(C_1\cdot ...\cdot C_n) \in [-\infty, +\infty[
$$
where the infimum is taken over all nef $\R$-Cartier $b$-divisors $C_i$ over $0$ such that $C_i\ge W_i$ for each $i$. 
\end{defi}
Note that $(W_1\cdot...\cdot W_n)$ is finite when all $W_i$ are bounded below. This is for instance the case if each $W_i$ is the nef envelope of a Cartier $b$-divisor by Lemma \ref{lem:bounded}. 

The next theorem summarizes the main properties of the intersection product. The non-trivial part of the assertion is additivity, which requires the approximation theorem. 

\begin{thm}\label{thm:inter}
The intersection product $(W_1,\dots,W_n) \mapsto W_1\cdot ...\cdot W_n$
of nef $\R$-Weil $b$-divisors over $0$ is 
symmetric, upper semi-continuous, and continuous along monotonic families (for the topology of coefficient-wise convergence). 

It is also homogeneous, additive, and non-decreasing in each variable.
Furthermore, $W_1\cdot ...\cdot W_n < 0$ if $W_i\neq 0$ for each $i$. \end{thm}

\begin{proof}
We follow the same lines as~\cite[Proposition~4.4]{BFJ}. Symmetry, homogeneity and monotonicity are clear. If $W_i \neq 0$ for all $i$ then there exists $\ep>0$ such that 
$W_i \le \ep Z(\frm)$ for all $i$ by Corollary~\ref{cor:izumi}, hence 
$$
W_1\cdot ...\cdot W_n  \le \ep^n Z(\frm)^n=-\ep^n e(\frm)<0
$$ 
where $e(\frm)$ is the Samuel multiplicity of $\frm$.

Let us prove the semi-continuity. Suppose that $W_i\neq0$ for all $i$, and pick $t\in\R$ such that
$W_1\cdot ...\cdot W_n  < t$. By definition there exist nef $\R$-Cartier $b$-divisors $C_i$ over $0$ such that $W_i\le C_i$ and $C_1\cdot ...\cdot  C_n   < t$. 
Replacing each $C_i$ by $(1-\ep)C_i$
we may assume $C_i \ne W_i$ while still preserving the previous conditions. 
Now consider the set $U_i$ of all nef $b$-divisors $W'_i$ such that
$W'_i \le C_i$. This is a neighborhood of $W_i$ in the topology of coefficient-wise convergence and $( W'_1\cdot ...\cdot W'_n )  < t$
for all $W'_i \in U_i$. This proves the upper semi-continuity. 

As a consequence we get the following continuity property: for all families $W_{j,k}$ such that 
\begin{itemize}
\item $W_{j,k}\ge W_j$ for all $j,k$ and
\item $\lim_k W_{j,k}=W_j$ for all $j$ 
\end{itemize}
we have $\lim_k W_{1,k}\cdot...\cdot W_{n,k}=W_1\cdot...\cdot W_n$. Indeed $W_{1,k}\cdot...\cdot W_{n,k}\ge W_1\cdot...\cdot W_n$ holds by monotonicity and the claim follows by upper semi-continuity.

We now turn to additivity. Assume first that $W'$, $W_1,W_2,...,W_n$ are nef envelopes of Cartier $b$-divisors over $0$. By Theorem \ref{thm:approxisol} there exist two sequences $C'_k$ and $C_{j,k}$ of nef Cartier divisors above $0$ such that $C_{j,k}\ge W_j$ and $C_{j,k}\to W_j$ as $k\to\infty$, and similarly for $C'_k$ and $W'$. Since $C_{1,k}+C'_k\ge W_1+W'$ also converges to $W_1+W'$ the above remark yields
$$
(C_{1,k}+C'_k)\cdot C_{2,k}\cdot ...\cdot C_{n,k}\to(W_1+W')\cdot W_2\cdot ...\cdot W_n
$$
On the other hand we have 
$$
(C_{1,k}+C'_k)\cdot C_{2,k}\cdot ...\cdot C_{n,k}=(C_{1,k}\cdot C_{2,k}\cdot ...\cdot C_{n,k})+(C'_k\cdot C_{2,k}\cdot ...\cdot C_{n,k})
$$
where 
$$
(C_{1,k}\cdot C_{2,k}\cdot ...\cdot C_{n,k})\to (W_1\cdot W_2\cdot ...\cdot W_n) \text{ and }(C_{1,k}\cdot C_{2,k}\cdot ...\cdot C_{n,k})\to (W'\cdot W_2\cdot ...\cdot W_n)
$$ 
so we get additivity for nef envelopes. 

In the general case let $W',W_1,W_2,...,W_n$ be arbitrary nef $b$-divisors over $0$. We then have $\Env_\p(W_{j,\p})\ge W_j$ and $\Env_\p(W_{j,\p})$ is a non-increasing net converging to $W_j$ by Remark~\ref{rmk:Env-comparison}. The additivity then follows from the previous case and the continuity along decreasing nets.

Finally, the continuity along non-decreasing sequences is the content
of Theorem~\ref{thm:increasing}, which is proven in the Appendix 
and will appear in a more general setting in \cite{BFJ11}.
\end{proof}

The expected local Khovanskii-Teissier inequality holds: 

\begin{thm}\label{thm:kt}
For all nef $\R$-Weil $b$-divisors $W_1,...,W_n$ over $0$ we have 
\begin{equation}\label{eq:minkowski}
|W_1\cdot ...\cdot W_n |\le
|W_1 ^n|^{1/n}\dots |W_n^n|^{1/n}. 
\end{equation}
In particular we have 
$$
|(W_1+W_2)^n|^{1/n}  \le |W_1^n|^{1/n}  +  |W_2^n|^{1/n} ~.
$$
\end{thm}

\begin{proof} Arguing as in the proof of Theorem \ref{thm:inter} we may use Theorem \ref{thm:approxisol} to reduce to the case where $W_i=Z(\fra_i)$ for some $\frm$-primary ideals $\fra_i$. In that case the result follows from Proposition \ref{prop:ram} and the local Khovanskii-Teissier inequality (cf.~\cite[Theorem~1.6.7 (iii)]{Laz1}). 
\end{proof}

\begin{prop}\label{prop:transform-volume}
Suppose $\f\colon (X,0) \to (Y,0)$ is a finite map of degree $e(\f)$.
Then for all nef $\R$-Weil $b$-divisors $W_1, \dots, W_n$ over $0\in Y$ we have:
\begin{equation}\label{eq:pull-back-int}
(\f^*W_1)\cdot...\cdot(\f^*W_n) = e(\f)\, W_1\cdot...\cdot W_n .
\end{equation}
\end{prop}

\begin{proof} Arguing as in the proof of Theorem \ref{thm:inter} by successive approximation relying on Theorem \ref{thm:approxisol}, we reduce to the case where each $W_j$ is $\R$-Cartier over $0$. Let $\p\colon Y'\to Y$ be a common determination of the $W_j$ which is an isomorphism away from $0$. Since $\phi^{-1}(0)=0$ there exists a birational morphism $\mu\colon X'\to X$ which is an isomorphism away from $0$ such that $\phi$ lifts as a morphism $\f'\colon X'\to Y'$, whose degree is still equal to $e(\phi)$ and the result follows. 
\end{proof}

\begin{rmk}\label{rmk:volume} 
For every graded sequence $\fra_\bullet$ of $\frm$-primary ideals we have 
$$
-Z(\fra_\bullet)^n= \lim_{k \to \infty} \frac{\dim_\C(\O_X/ \fra_k)}{k^n/n!}.
$$
Indeed it was shown by Lazarsfeld and Musta\c{t}\v{a}  \cite[Theorem 3.8]{LM} that the right-hand side limit exists and coincides with $\lim_{k\to\infty} e(\fra_k)/k^n$ (which corresponds to a local version of the Fujita approximation theorem). On the other hand $Z(\fra_\bullet)$ is the non-decreasing limit of $\tfrac{1}{k!} Z(\fra_{k!})$ hence 
$Z(\fra_\bullet)^n=\lim_{k\to\infty}Z(\fra_k)^n/k^n$ by using the continuity of intersection numbers along non-decreasing sequence 
%(cf. Remark \ref{rem:bfj}) 
and the claim follows in view of Proposition \ref{prop:ram}. 
\end{rmk}

\subsection{The volume of an isolated singularity}\label{sec:defvol} 
By Proposition \ref{prop:bounded} the log-discrepancy divisor $A_{\X/X}$ is always bounded below. Its nef envelope $\Env_\X(A_{\X/X})$ is therefore well-defined and bounded below as well, and we may introduce:
\begin{defi}
The {\it volume} of a normal isolated singularity $(X,0)$ is defined as
$$
\Vol(X,0) :=-\Env_\X\left(A_{\X/X}\right)^n~.
$$
\end{defi}

We have the following characterization of singularities with zero volume:
\begin{prop}\label{prop:volzero} $\Vol(X,0)=0$ iff $A_{\X/X}\ge 0$. When $X$
 is $\Q$-Gorenstein,  $\Vol(X,0)=0$ iff  it has log-canonical singularities. 
\end{prop}
\begin{proof} By Theorem \ref{thm:inter} we have $\Vol(X,0)=0$ iff $\Env_\X(A_{\X/X})=0$, which is equivalent to $A_{\X/X}\ge 0$ since every $X$-nef $b$-divisor over $0$ is antieffective by the negativity lemma. 

When $X$ is $\Q$-Gorenstein, then $A_{\X/X} = A_{m,\X/X}$ for any integer $m$
such that $mK_X$ is Cartier. We conclude recalling that 
$X$ is log-canonical if the trace of the
log-discrepancy divisor $A_{m,\X/X}$ in one (or equivalently any) log-resolution of $X$ is effective. 
\end{proof}

The volume satisfies the following basic monotonicity property:

\begin{thm}\label{thm:volume} Let $\phi:(X,0)\to(Y,0)$ be a finite morphism between normal isolated singularities. Then we have 
$$
\Vol(X,0)\ge e(\phi)\Vol(Y,0),
$$
with equality if $\phi$ is \'etale in codimension $1$. 
\end{thm}
\begin{proof} 
We have $A_{\X/X}\le\phi^*A_{\Y/Y}$ by Corollary \ref{cor:jacrel}, and equality holds if and only if $R_\phi=0$, i.e.~iff $\phi$ is  \'etale in codimension $1$. The result follows immediately using Theorem~\ref{prop:pull-env} and Proposition~\ref{prop:transform-volume}. 
\end{proof}

\subsection{The volume of a cone singularity}
In the case of a cone singularity, the volume relates to the positivity of the anticanonical divisor of the exceptional divisor in the following way. 

\begin{prop}\label{prop:conevol} 
Let $0\in X$ be the affine cone over a polarized smooth variety $(V,L)$ as in Example \ref{eg:cone}. We assume in particular that $X$ is normal.
\begin{enumerate}
\item
If $|-mK_V|$ contains a smooth element for some $m \ge 1$, then $\Vol(X,0)=0$.
\item
Conversely, if $\Vol(X,0)=0$ then $-K_V$ is pseudoeffective. 
\end{enumerate}
\end{prop}

\begin{proof}
Denote by $\p\colon  X_\p\to X$ the blow-up at $0$, with exceptional divisor $E\simeq V$. 
If $D \in |-mK_V|$ is a smooth element, then we consider the pair $(X,\D)$ where
$\D$ is the cone over $D$ divided by $m$. Note that $\p$
gives a log resolution of $(X,\D)$ and $K_{X_\p} + E - \p^*(K_X + \D)$ has order one along 
$E$, by adjunction. Therefore $(X,\D)$ is log canonical, hence $A_{m,\X/X}\ge 0$. This implies that $A_{\X/X} \ge 0$, 
and thus $\Vol(X,0)=0$ by Proposition~\ref{prop:volzero}.

Conversely, assume that $\Vol(X,0) = 0$. 
We then have $a=\ord_E(A_{\X/X})\ge 0$ by Proposition \ref{prop:volzero} and
$$
K_{X_\p}+ E + \Env_X(-K_X)_\p =aE
$$
since $E$ is the only $\p$-exceptional divisor. Now $\Env_X(-K_X)_\p$ restricts to a pseudoeffective class in $N^1(E)$ by Lemma \ref{lem:mov}. The pseudoeffectivity of $-K_E$ follows by adjunction (one can also see that $-K_E$ is big if the `generalized log-discrepancy' $a$ is positive). 
\end{proof}

In~\cite[Chapter~2, Example~55]{Kol11} Koll\'ar gives an example
of a family of singular threefolds where the central fiber admits a boundary
which makes it into a  log canonical pair while the nearby fibers do not.
The same kind of example can be used to show that the volume defined above is not a topological invariant of the \emph{link} of the singularity in general, in contrast with the $2$-dimensional case. We are grateful to J\'anos Koll\'ar for bringing this example to our attention. 

Recall first that a link $M$ of an isolated singularity $0\in X$ is a compact real-analytic hypersurface of $X\setminus\{0\}$ with the property that $X$ is homeomorphic to the (real) cone over $M$. It can be constructed as follows (cf. for example \cite[Section 2A]{Loo}). Let $r\colon X\to\R_+$ be a real analytic function defined in a neighborhood of $0$ such that $r^{-1}(0)=\{0\}$ (for instance the restriction to $X$ of $\|z\|^2$ in a local analytic embedding in $\C^N$). Upon shrinking $X$ we may assume that $r$ has no criticial point on $X\setminus\{0\}$, and $M$ can then be taken to be any level set $r^{-1}(\ep)$ for $0<\ep\ll 1$. 

If $0\in X$ is the affine cone over a polarized variety $(V,L)$ then its link $M$ is diffeomorphic to the (unit) circle bundle of any Hermitian metric on $L^*$. Indeed we may take the function $r$ to be given by $r(v)=\sum_j|\langle s_j,v^m\rangle|^{2/m}$ where $(s_j)$ is a basis of sections of $mL$ for $m\gg 1$. As a consequence, the links of the cone singularities $X_t$ induced by any smooth family of polarized varieties $(V_t,L_t)_{t\in T}$ are all diffeomorphic - as follows by applying the Ehresmann-Feldbau theorem to the family of circle bundles with respect to a Hermitian metric on $L$ over the total space of the family $V_t$. 

We will use the following result.
\begin{lem}\label{lem:psef} Let $S_r$ be the blow-up of $\P^2$ at $r$ very general points. Then $-K_{S_r}$ is not pseudo-effective if (and only if) $r\ge 10$. 
\end{lem}
This fact is certainly well-known to experts, but we provide a proof for the convenience of the reader. 
\begin{proof} By semicontinuity it is enough to show that $-K_{S_r}$ is not pseudoeffective for the blow-up $S_r$ of $\P^2$ at some family of $r\ge 10$ points. We may also reduce to the case $r=10$ since the anticanonical bundle only becomes less effective when we keep blowing-up points. 

First, by \cite[Lemma 3.1]{Sak84}, for any rational surface $S$ we have $-K_S$ pseudoeffective iff $h^0(-mK_S)>0$ for some positive integer $m$. The short proof goes as follows. The non-trivial case is when $-K_S$ is pseudoeffective but not big. Let $-K_S=P+N$ be the Zariski decomposition, which satisfies $P^2=P\cdot K_S=0$. By Riemann--Roch it follows that $\chi(mP)=\chi(\O_S)=1$ for any $m$ such that $mP$ is Cartier. But $h^2(mP)=h^0(K_S-mP)=0$ because $K_S$ is not pseudoeffective, hence $h^0(mP)\ge\chi(mP)$, and the result follows. 

Second, let $S_9$ be the blow-up of $\P^2$ at 9 very general points $p_i$ of a given smooth cubic curve $C$ with inflection point $p$. We then have $h^0(-mK_{S_9})=1$ for all positive integers $m$, otherwise we would get by restriction to the strict transform of $C$  $H^0\left(\O_C(3m)(-m\sum_i p_i)\right)\neq 0$, and $9p-\sum_i p_i$ would be $m$-torsion in $\Pic^0(C)\simeq C$. In other words, we see that $mC$ is the only degree $3m$ curve in $\P^2$ passing through each $p_i$ with multiplicity at least $m$. If we let $p_{10}$ be any point outside $C$ it follows of course that no degree $3m$ curve passes through $p_1,...,p_{10}$ with multiplicity at least $m$. But this means that the blow-up $S_{10}$ of $\P^2$ at $p_1,...,p_{10}$ has $h^0(-mK_{S_{10}})=0$ for all $m$, so that $-K_{S_{10}}$ is not pseudoeffective by Sakai's lemma. 
\end{proof}

We are now in a position to state our example. 
\begin{eg}\label{eg:not-top-inv}
Let $T$ be the parameter space of all sets of $r$ distinct points $\Sigma_t\subset \P^2$, and for each $t\in T$ let $V_t$ be the blow-up of $\P^2$ at $\Sigma_t$. Let $L$ be a polarization of the smooth projective family $(V_t)_{t\in T}$ and let $X_t$ be the associated family of cone singularities, whose links are all diffeomorphic according to the above discussion. 
After possibly replacing $L$ by a multiple, we can assume that each $X_t$ is normal.

If for a given $t\in T$ the points $\Sigma_t$ all lie on a smooth cubic curve then the anticanonical system $|-K_{V_t}|$ contains the strict transform of that curve, and we thus have $\Vol(X_t,0)=0$ for such values of $t$ by Proposition \ref{prop:conevol}. On the other hand Proposition \ref{prop:conevol} and Lemma \ref{lem:psef}Ê show that $\Vol(X_t,0)>0$ for $t\in T$ very general. 
\end{eg}

%
%%%%%%%%%%%%%%%%%%%%%%
%

\section{Comparison with other invariants of isolated singularities}\label{sec:compare}
\subsection{Wahl's characteristic number}
As recalled in the introduction, Wahl defined in~\cite{Wahl} the {\it characteristic number} of a normal surface singularity $(X,0)$ as $-P^2$ of the nef part $P$ in the Zariski decomposition of $K_{X_\p} + E$, where $\p\colon X_\p\to X$ is any 
log-resolution of $(X,0)$ and $E$ is the reduced exceptional divisor of $\p$. 
The following result proves that the volume defined above extends 
Wahl's invariant to all isolated normal singularities. 

\begin{prop}\label{prop:wahl}
If $(X,0)$ is a normal surface singularity then 
$\Vol(X,0)$ coincides with Wahl's characteristic number. 
\end{prop}

\begin{proof} Let $\p\colon X_\p\to X$ be log-resolution of $(X,0)$ and let $E$ be its reduced exceptional divisor. By Theorem \ref{thm:mumford} we see that $\Env_\p(A_{X_\p/X})$ coincides with the nef part of $K_{X_\p}+E-\p^*K_X$. Since the latter is $\p$-numerically equivalent to $K_{X_\p}+E$ it follows that $\Env_\p(A_{X_\p/X})$ is $\p$-numerically equivalent to the nef part $P$ of $K_{X_\p}+E$, so that 
$$
-P^2=-\Env_\p(A_{X_\p/X})^2.
$$
On the other hand we claim that $\Env_\p(A_{X_\p/X})=\Env_\X(A_{\X/X})$, which will conclude the proof. Indeed on the one hand we have
$$
\Env_\X(A_{\X/X})\le\Env_\p(A_{X_\p/X})
$$
as for any Weil $b$-divisor. On the other hand Lemma \ref{lem:increase} implies that 
$$
K_\X+1_\X\ge\overline{K_{X_\p}+E}
$$
over $0$, hence $A_{\X/X}\ge\overline{A_{X_\p/X}}$, and we infer 
$\Env_\X(A_{\X/X})\ge\Env_\p(A_{X_\p/X})$ as desired.  
\end{proof}

\begin{proof}[Proof of Theorem~A]
The definition of the volume is given in \S\ref{sec:defvol}.
Theorem~A (i) is precisely Theorem~\ref{thm:volume}.
Statement (ii) is Proposition~\ref{prop:wahl}.
Statement (iii) is Proposition~\ref{prop:volzero}.
\end{proof}

%%%%

\subsection{Plurigenera and Fulger's volume}\label{sec:plurigenera}
Let $0\in X$ be (a germ of) an isolated singularity and let $\p\colon X_\p\to X$ be a log-resolution with reduced exceptional SNC divisor $E$. One may then consider the following plurigenera (see \cite{Ish2} for a review).
\begin{itemize}
\item Kn\"oller's plurigenera \cite{Kno}, defined by
$$
\gamma_m(X,0):=\dim H^0(X_\p \!\setminus\! E,mK_{X_\p})/H^0(X_\p,mK_{X_\p}).
$$
\item Watanabe's $L^2$-plurigenera \cite{Wat1}, defined by
$$
\delta_m(X,0):=\dim H^0(X_\p \!\setminus\! E,mK_{X_\p})/H^0(X_\p,mK_{X_\p}+(m-1)E).
$$
\item Morales' log-plurigenera \cite[Definition 0.5.4]{Mora}, defined by 
$$
\lambda_m(X,0):=\dim H^0(X_\p \!\setminus\! E,mK_{X_\p})/H^0(X_\p,m(K_{X_\p}+E)).
$$
\end{itemize}
These numbers do not depend on the choice of log-resolution. 
They satisfy 
$$
\lambda_m(X,0)\le\delta_m(X,0)\le\gamma_m(X,0) = O(m^n),
$$ 
and one may use them to define various notions of Kodaira dimension of an isolated 
singularity.

In a recent work, Fulger~\cite{Fulg} has explored in more detail the growth of these numbers.
His framework is the following.
Given a Cartier divisor $D$ on $X_\p$,  consider the local cohomological dimension 
$$
h^1_{\{0\}}(D)=\dim H^0(X_\p\!\setminus\! E,D)/H^0(X_\p,D)=\dim\O_X(\p_*D)/\O_X(D).
$$
Observe that $\gamma_m(X,0)=h^1_{\{0\}}(mK_{X_\p})$ and $\lambda_m(X,0)=h^1_{\{0\}}(m(K_{X_\p}+E))$. Fulger proves that $h^1_{\{0\}}(mD)=O(m^n)$ and defines
the \emph{local volume} of $D$ by setting 
$$
\vol_{\loc}(D):=\limsup_{m\to\infty}\frac{n!}{m^n}h^1_{\{0\}}(mD)~.
$$
When the Cartier divisor $D$ lies over $0$ one has:

\begin{prop}\label{prop:compare}
Suppose $D$ is a Cartier divisor in $X_\pi$ lying over $0$. 
Then 
$$
\vol_{\loc}(D)=-\Env_\X(\overline D)^n~.
$$
\end{prop}
\begin{proof}
We may assume $D\le 0$. The envelope of $D$ is the $b$-divisor associated to the graded sequence of $\mathfrak{m}$-primary ideals $\O_X(-mD)$. The result follows from  Remark~\ref{rmk:volume}.
\end{proof}

Fulger \cite{Fulg} then introduces an alternative notion of volume of an isolated singularity by setting:
$$
\Vol_F(X,0) := \vol_{\loc} (K_{X_\p} + E).
$$ 

\begin{prop}\label{p:vol-Q-gor}
$\Vol(X,0)=\Vol_F(X,0)$ if $X$ is $\Q$-Gorenstein.
\end{prop}

\begin{proof}
For any  integer $m$ such that $mK_X$ is Cartier, one has
$A_{\X/X} = A_{m,\X/X}$. Pick any log-resolution $\p\colon  X_\p \to X$.
Then Lemma~\ref{lem:increase}  applied to $X_\p$ shows that $\ov{A_{X_\p/X}} \le A_{\X/X}$. In particular, these $b$-divisors share the same envelope. 
We conclude by Proposition~\ref{prop:compare} above. 
\end{proof}

In general, Fulger proves that there is always an inequality
$$
\Vol(X,0) \ge \Vol_F(X,0).
$$
We know by \cite{Wahl} that in dimension two these volumes always coincide. 
In higher dimension these two invariants may however differ, as shown by the following example. 

\begin{eg} \label{eg:counterfulger}
Let $V$ be any smooth projective variety such that neither $K_V$ nor $-K_V$ are pseudoeffective, for instance $V=C\times\P^1$ where $C$ is a curve of genus at least $2$. Pick any ample line bundle $L$ on $V$ such that the the affine cone $0\in X$ over $(V,L)$ is normal. We claim that 
$$
\Vol(X,0)>0=\Vol_F(X,0).
$$
Indeed, Proposition~\ref{prop:conevol} and the fact that $-K_V$ is not pseudoeffective show that $\Vol(X,0)>0$. On the other hand, the fact that $K_V$ is not pseudoeffective implies that $\delta_m(X,0)=0$ for al $m$, hence  $\Vol_F(X,0)=0$. To see this, let $\p\colon X_\p\to X$ be the blow-up of $0$, with exceptional divisor $E\simeq V$. Since $L$ is ample, $mK_V-(p-m)L$ is not pseudoeffective for any $p\ge m$, hence
$$
H^0(E,mK_E+(p-m)E|_E)\simeq H^0(V,mK_V-(p-m)L)=0
$$
Now $(K_{X_\p}+E)|_E=K_E$ by adjunction, and the restriction morphism
$$
H^0(X_\p,mK_{X_\p}+pE)/H^0(X_\p,mK_{X_\p}+(p-1)E)\to H^0(E,mK_E+(p-m)E|_E)
$$
is injective.  We have thus shown 
$H^0(X_\p,mK_{X_\p}+(m-1)E)=H^0(X_\p,mK_{X_\p}+pE)$ for all $p\ge m$, hence 
$H^0(X_\p,mK_{X_\p}+(m-1)E)=H^0(X_\p\setminus E, mK_{X_\p})$, i.e.~$\delta_m(X,0)=0$. 
\end{eg}

%
%%%%%%%%%%%%%%%%%%%%%%
%

\section{Endomorphisms}\label{sec:endo}
We apply the previous analysis to the study of normal isolated singularities admitting endomorphisms.

\subsection{Proofs of Theorems~B and~C}
We start by  proving the following result. 
\begin{thm}\label{thm:klt} Assume that $X$ is numerically Gorenstein and let $\phi:(X,0)\to (X,0)$ is a finite endomorphism of degree $e(\f)\ge 2$ such that $R_\phi\neq 0$. Then there exists $\ep>0$ such that $A_{\X/X}\ge -\ep Z(\frm)$.
\end{thm}
\begin{rmk} \label{rem:klt}
When $X$ is $\Q$-Gorenstein or $\dim X=2$, the condition  $A_{\X/X}\ge -\ep Z(\frm)$ for some $\ep>0$ is equivalent to $A_{m,\X/X}>0$ for some $m$. By Corollary \ref{cor:classical} the latter condition means in turn that $X$ has klt singularities in the sense  that there exists a $\Q$-boundary $\D$ such that $(X,\D)$ is klt. 
It is possible to prove this result unconditionnally; we shall return to this problem
in a later work.
\end{rmk}

\begin{rmk}
Tsuchihashi's cusp singularities (see below) show that the assumption $R_\phi\neq 0$ is essential even when $K_X$ is Cartier. 
\end{rmk}

\begin{proof} Since $X$ is numerically Gorenstein $R_{\phi^k}=K_X-(\phi^k)^*K_X$ is numerically Cartier for each $k$ and Corollary \ref{cor:jacrel} yields 
$$
(\phi^k)^*A_{\X/X}=A_{\X/X}+\Env_X(R_{\phi^k}). 
$$
On the other hand observe that 
$R_{\phi^k}=\sum_{j=0}^{k-1}(\phi^j)^*R_\phi$ by the chain-rule. Each $(\phi^j)^*R_\phi$ is numerically Cartier as well, so that 
$$
\Env_X(R_{\phi^k})=\sum_{j=0}^{k-1}(\phi^j)^*\Env_X(R_\phi)
$$
by Lemma \ref{lem:additive} and Proposition \ref{prop:pull-env}. Using Proposition \ref{prop:bounded} and Theorem \ref{thm:uniform} we thus obtain $c_1,c_2>0$ such that 
$$
(\phi^k)^*(A_{\X/X})\ge c_1Z(\frm)-c_2\sum_{j=0}^{k-1}(\phi^j)^*Z(\frm)
$$
for all divisorial valuations $\nu$ centered at $0$ and all $k$. Since we have $(\phi^j)^*\frm\subset\frm$ it follows that 
$$
(\phi^k)^*A_{\X/X}\ge -Z(\frm)(kc_2-c_1).
$$
But the action of $\phi^k$ on divisorial valuations centered at $0$ is surjective by Lemma \ref{lem:surjdiv}. We furthermore have $\nu\left((\phi^k)^*A_{\X/X}\right)=\nu\left((\phi^k)^*\frm\right)\nu\left(A_{\X/X}\right)$ for each divisorial valuation $\nu$ centered at $0$ and there exists $c_k>0$ such that $\nu((\phi^k)^*\frm)\le c_k\nu(\frm)$ for all $\nu$ by Lemma \ref{lem:bounded}. We thus get
$A_{\X/X}\ge -\ep_k Z(\frm)$ with 
$$
\ep_k:=\frac{kc_2-c_1}{c_k}>0
$$
as soon as $k>c_1/c_2$. 
\end{proof}

\begin{proof}[Proof of Theorem~B]
If $\phi \colon  X \to X$ is a finite endomorphism with $e(\phi) \ge 2$, then Theorem~A implies
$\Vol(X,0) \ge 2 \Vol(X,0)$ hence $\Vol(X,0) =0$. 
When $X$ is $\Q$-Gorenstein and $\f$ is not \'etale in codimension $1$, then $X$ is klt by the previous theorem and Remark~\ref{rem:klt}.
\end{proof}

\begin{proof} [Proof of Theorem~C]
By assumption, there exists an endomorphism $\f\colon  V \to V$ and an ample line bundle $L$ such that  $\f^* L\simeq dL$ for some $d \ge 2$. 
The composite map
$$
H^0(V,mL) \mathop{\to}\limits^{\f^*} H^0(V,m\f^*L)\simeq H^0(V, dmL)
$$
induces an endomorphism of the finitely generated algebra $\bigoplus_{m\ge 0}H^0(V,mL)$ (which does not preserve the grading). Since the spectrum of this algebra is equal to $X=C(V)$, we get an induced endomorphism $C(\f)$ on $C(V)$. It is clear that $C(\f)$ is finite, fixes the vertex $0\in X$, and  is not an automorphism. We conclude that $\Vol(X,0)=0$, which implies that $-K_V$ is pseudoeffective by Proposition \ref{prop:conevol}. 
\end{proof}

\subsection{Simple examples of endomorphisms.}\label{sec:simple}
A quotient singularity is locally isomorphic to $(\C^n /G, 0)$ where $G$ is a finite group acting linearly on $\C^n$. Let $\pi\colon  \C^n \to \C^n/G$ be the natural projection. For any holomorphic maps $h_1, ..., h_n\colon  \C^n/G \to \C$ such that $\cap h_i^{-1}(0) = (0)$,  the composite map $ \pi \circ (h_1, ... h_n) \colon  (\C^n/G,0) \to (\C^n/G,0)$
is a finite endomorphism  of degree $\ge 2$ if the singularity is non trivial. Note also
that any toric singularity admits finite endomorphisms of degree $\ge2$  (induced by the multiplication by an integer $\ge2$ on its associated fan). 

We saw above examples of endomorphisms
on cone singularities. One can modify this construction to get examples on other kind of
simple singularities. 

Consider a smooth projective morphism $f\colon Z\to C$ to a smooth pointed curve $0\in C$ and suppose given a non-invertible endomorphism $\phi$ such that $f\circ\phi=f$. Note that $\phi$ is automatically finite since the injective endomorphism $\phi^*$ of $N^1(Z/C)$ has to be bijective.

Assume that $D\subset Z_0$ is a smooth irreducible ample divisor of the fiber $Z_0$ over $0$ that does not intersect the ramification locus of $\phi$ and such that $\phi(D)\subset D$. Denote  by $Y\to Z$ be the blow-up of $Z$ along $D$. Then $\phi$ lifts to a rational self-map of $Y$ over $C$, and the fact that $\phi$ is \'etale around $D$ implies that the indeterminacy locus of this rational lift is contained in $\mu^{-1}(\phi^{-1}(D)\setminus D)$ hence in the strict transform $E$ of $Z_0$ on $Y$. 

Since the conormal bundle of $E$ in $Y$ is ample,  $E$ contracts to a simple singularity $0\in X$ by \cite{Gra} (we are therefore dealing with an analytic germ $0\in X$ in that case). The above discussion shows that $\phi$ induces a finite endomorphism of $(X,0)$, which is furthermore not invertible since $\phi$ was assumed not to be an automorphism. 

Basic examples of this construction include deformations of abelian varieties
having a section, with $\phi$ the multiplication by a positive integer.

\subsection{Endomorphisms of cusp singularities}\label{sec:cusp}

Our basic references are~\cite{oda,tsu}. Let $C\subset\R^n$ be an open convex cone that is strongly convex (i.e.~its closure contains no line) and let $\Gamma\subset\mathrm{SL}(n,\Z)$ be a subgroup leaving $C$ invariant, whose action on $C/\R^*_+$ is
properly discontinuous without fixed point, and has compact quotient. Denote by 
$$M:=\Gamma\backslash C/\R_+^*$$ 
the corresponding $(n-1)$-dimensional orientable manifold. 

Consider the convex envelope $\Theta$ of $C\cap\Z^n$. It is proved in \cite{tsu} that the faces of $\overline\Theta$ are convex polytopes contained in $C$ and with integral vertices. Since $\Theta$ is $\Gamma$-invariant the cones over the faces of $\Theta$ therefore give rise to a $\Gamma$-invariant rational fan $\Sigma$ of $\R^n$ with $|\Sigma|=C\cup\{0\}$. This fan is infinite but is finite modulo $\Gamma$ since $M$ is compact. 

The (infinite type) toric variety $X(\Sigma)$ comes with a $\Gamma$-action which preserves the toric divisor $D:=X(\Sigma)\setminus (\C^*)^n$ as well the inverse image of $C$ by the map $\mathrm{Log}:(\C^*)^n\to\R^n$ defined by 
$$
\mathrm{Log}(z_1,...,z_n)=(\log|z_1|,...,\log|z_n|).
$$
The $\Gamma$-invariant set $U:=\mathrm{Log}^{-1}(C)\cup D$ is open in $X(\Sigma)$ and the action of $\Gamma$ is properly discontinuous and without fixed point on $U$. One then shows that the divisor $E:=D/\Gamma\subset U/\Gamma=:Y$, which is compact since $\Sigma$ is a finite fan modulo $\Gamma$, admits a strictly pseudoconvex neighbourhood in $Y$, so that it can be contracted to a normal singularity $0\in X$, which is furthermore isolated since $Y-E$ is smooth. Note that $Y$, though possibly not smooth along $E$, has at most rational singularities since $U$ does, being an open subset of a toric variety. The isolated normal singularity $(X,0)$ is called the \emph{cusp singularity} attached to $(C,\Gamma)$. It is shown in \cite{tsu} that $(C,\Gamma)$ is determined up to conjugation in $\mathrm{GL}(n,\Z)$ by the (analytic) isomorphism type of the germ $(X,0)$. 

\begin{lem}
The canonical divisor $K_X$ is Cartier, $X$ is lc but not klt. 
\end{lem}
\begin{rmk}
Cusp singularities are however not Cohen-Macaulay in general, hence not Gorenstein.
\end{rmk}
\begin{proof}
The $n$-form $\Om= \frac{dz_1}{z_1} \wedge ... \wedge \frac{dz_n}{z_n}$
on the torus $(\C^*)^n$ extends to $X(\Sigma)$ with poles of order one along $D$. It is $\Gamma$-invariant since $\Gamma$ is a subgroup of $\mathrm{SL}(n,\Z)$ thus it descends to a meromorphic form on $U/\Gamma$ with order one poles along $D/\Gamma$. We conclude $K_X$ is zero and that $X$ is lc but not klt since $\pi\colon (Y,E)\to X$ is crepant and $(X(\Sigma),D)$ is lc but not klt as for any toric variety. 
\end{proof}

Now let $A\in\mathrm{GL}(n,\R)$ with integer coefficient which preserves $C$ and commutes with $\Gamma$ (e.g.~~a homothety). 
Then $Z$ induces a regular map on $U$ that descends to the quotient
$Y$ and preserves the divisors $E$ and we get a finite endomorphism
$\phi: (X,0) \to (X,0)$ whose topological degree is equal to $|\det A|$.

\begin{eg}[Hilbert modular cusp singularities]
Let $K$ be a totally real number field of degree $n$ over $\Q$ and let $N$ be a free $\Z$-submodule of $K$ of rank $n$ (for instance $N=\O_K$). Using the $n$ distinct embeddings of $K$ into $\R$ we get a canonical identification $K\otimes_\Q \R=\R^n$ and we may view $N$ as a lattice in $\R^n$. Now set $C:=(\R^*_+)^n\subset N_\R$ and consider the group $\Gamma^+_N$ of totally positive units of $u\in\O_K^*$ such that $uN =N$, where $u$ is said to be totally positive if its image under any embedding of $K$ in $\R$ is positive. By Dirichlet's unit theorem, $\Gamma^+_N$ is isomorphic to $\Z^{n-1}$, and there is a canonical
injective homomorphism $\Gamma^+_N \hookrightarrow \mathrm{SL}(N)$.
For any subgroup $\Gamma\subset \Gamma^+_N$ of finite index, the triple $(N,C,\Gamma)$ then satisfies the requirements of the definition of a cusp singularities. The singularities obtained by this construction are called Hilbert modular cusp singularities.
\end{eg}

\appendix
\section{Continuity of intersection products along non-decreasing nets}

In this appendix, we fix an isolated normal singularity $0\in X$ as in Section~\ref{sec:isolated}.
The following theorem is taken from \cite{BFJ11}, where the result will appear in
a more general form. We are very grateful to Mattias Jonsson for allowing us
to include a proof here.

\begin{thm}[Increasing limits]\label{thm:increasing}
  For $1\le r\le n$,  let $\{ W_{r,i}\}_{i\in I}$ 
  be a net of nef $\R$-Weil $b$-divisors over $0$ increasing to $W_r$. Assume there exists some constant
  $C>0$ such that  $W_{r,i} \ge C Z(\frm)$ for all $r,i$.
  Then we have
  $$
  W_{1,i} \cdot ... \cdot W_{n,i} \to 
  W_{1} \cdot ... \cdot W_{n} ~. 
  $$
 \end{thm}

\begin{proof}
  After rescaling, we may assume $W_{r,i}\ge Z(\frm)$ for all $r,i$.
  We will prove the statement by induction 
  on $p=0,\dots,n-1$ under the assumption that
  $W_{r,i}=W_r$ for all $i$ and all $r>p$.
  
  The case $p=0$ is trivial, so first suppose $p=1$.
  Let $C_2,\dots,C_n$ be nef $\R$-Cartier $b$-divisors such that $C_r \ge W_r$ for
  $2 \le r \le n$. It follows from Lemma~\ref{C101} that
  \begin{align*}
    0
    \le {}& W_1\cdot ... \cdot W_n -W_{1,i}\cdot W_2 \cdot ... \cdot W_n\\
    = {}&-(W_1\cdot C_2 \cdot ... \cdot C_n- W_1 \cdot W_2\cdot ... \cdot W_n)\\
    &+ (W_1-W_{1,i})\cdot C_2\cdot ... \cdot C_n\\
    &+W_{1,i}\cdot C_2\cdot ... \cdot C_n- W_{1,i}\cdot W_2\cdot ... \cdot W_n\\
    \le {}&(W_1-W_{1,i}) \cdot C_2\cdot ... \cdot C_n
    +\sum_{r=2}^n \left( (C_r-W_r) \cdot W_r\cdot ... \cdot W_r\right)^{\frac1{2^{n-1}}}.
  \end{align*}
  Fix $\ep>0$. We can assume that the $b$-divisors $C_r$ are chosen such that
  $0\le (C_r-W_r) \cdot W_r\cdot ... \cdot W_r\le\epsilon$.
  On the other hand, since $C_r$ are $\R$-Cartier $b$-divisors and  $W_{1,i} \to W_1$, 
  we have $(W_1-W_{1,i}) \cdot C_2\cdot ... \cdot C_n \le\epsilon$
  for $i$ large enough. 

  Now assume $1<p<n$ and that the statement is true for $p-1$.
  Write 
  \begin{equation*}
    a_i= W_{1,i}\cdot ... \cdot W_{p,i} \cdot W_{p+1}\cdot W_n~.
  \end{equation*}
  Clearly $a_i$ is increasing in $i$ and we must show that 
  $\sup_i a_i= W_1\cdot ... \cdot W_n$.
  If $j\le i$, then $W_{p,j}\le W_{p,i}\le W_p$, and so 
  \begin{equation*}
    W_{1,i}\cdot ... \cdot W_{p-1,i}\cdot W_{p,j}\cdot W_{p+1}\cdot ... \cdot W_n
    \le a_i
\le     W_{1,i}\cdot ... \cdot W_{p-1,i}\cdot W_{p}\cdot W_{p+1}\cdot ... \cdot W_n.
  \end{equation*}
  Taking the supremum over all $i$, we get by the inductive assumption
  that 
  \begin{equation*}
    W_{1}\cdot ... \cdot W_{p-1}\cdot W_{p,j}\cdot W_{p+1}\cdot ... \cdot W_n
    \le \sup_i a_i \le
   W_1\cdot ... \cdot W_n.
  \end{equation*}
  The inductive assumption implies that the supremum over $j$
  of the first term equals $W_1\cdot ... \cdot W_n$.
  Thus $\sup_ia_i=W_1\cdot ... \cdot W_n$, which completes the proof.
\end{proof}

\begin{lem}[Hodge Index Theorem]\label{hodge}
 Let $Z_3, \cdots, Z_n$ be nef $\R$-Cartier $b$-divisors over $0$.
 Then
 $$
(Z,W) := Z \cdot W \cdot Z_3\cdot ... \cdot Z_n 
 $$
 defines a bilinear form on the space of Cartier $b$-divisors over $0$
 that is negative semidefinite.
\end{lem}
\begin{proof}
By choosing a common determination $Y \to X$, we are reduced to prove this statement
for exceptional  divisors lying in $Y$. We may perturb $Z_i$ and assume they are rational and ample over $0$. By intersecting by general elements of multiples of $Z_i$, we are then reduced to the two-dimensional case. 
Since the intersection form on the exceptional components of any birational surface map is negative definite, the result follows.
\end{proof}
\begin{lem}\label{L101}
 If $Z,W,Z_2,...,Z_n$ are nef $\R$-Weil $b$-divisors  over $0$  with 
 $Z(\frm) \le Z\le W \le 0$ and $Z(\frm) \le Z_j\le 0$ for $j\ge2$,
 then
 \begin{equation*}
   0\le  (W-Z)\cdot Z_2 \cdot ... \cdot Z_n
   \le \left( (W-Z) \cdot Z \cdot ... \cdot Z\right)^{\frac1{2^{n-1}}},
 \end{equation*}
\end{lem}
\begin{proof}
  We may assume that all the $b$-divisors involved are $\R$-Cartier. 
  By Lemma~\ref{hodge}, the bilinear form 
  $(Z,W) \mapsto Z\cdot W\cdot Z_3\cdot ... \cdot Z_n$ is negative semidefinite.
  Hence 
  \begin{multline*}
  0
  \le (W-Z)\cdot Z_2\cdot... \cdot Z_n
  \le |Z_2\cdot Z_2\cdot Z_3\cdot ... \cdot Z_n)|^{1/2}
  \. |(W-Z)\cdot (W-Z)\cdot Z_3 \cdot ... \cdot Z_n|^{1/2}\\
  \le |(W-Z)\cdot (W-Z)\cdot Z_3 \cdot ... \cdot Z_n|^{1/2}
   \le |(W-Z)\cdot Z\cdot Z_3 \cdot ... \cdot Z_n|^{1/2}.
\end{multline*}
Repeating this procedure $n-2$ times, we conclude the proof.
\end{proof}
\begin{lem}\label{C101}
 If $Z_r, W_r$ are nef $\R$-Weil $b$-divisors with 
 $Z(\frm)\le Z_r\le W_r\le 0$ for $1\le r\le n$, then
\begin{equation*}
   0
   \le W_1\cdot ... \cdot W_n  - Z_1 \cdot ... \cdot Z_n
   \le\sum_{r=1}^n \left( (W_r-Z_r)\cdot Z_r\cdot... \cdot Z_r \right)^{\frac1{2^{n-1}}}.
\end{equation*}
\end{lem}
\begin{proof}
It follows from Lemma~\ref{L101} by writing
\begin{multline*}
W_1\cdot ... \cdot W_n  - Z_1 \cdot ... \cdot Z_n
=
(W_1 -Z_1) \cdot W_2 \cdot... \cdot W_n
+\\
Z_1 \cdot (W_2 - Z_2) \cdot W_3 \cdot ... \cdot W_n
+ ...
+ 
Z_1\cdot ... \cdot Z_{n-1} \cdot (Z_n-W_n) \qedhere
\end{multline*}
\end{proof}


\begin{thebibliography}{BCHM10}

\bibitem[BCHM10]{BCHM} C. Birkar, P. Cascini, C. Hacon, J. McKernan.
{\it Existence of minimal models for varieties of log general type}.  
J. Amer. Math. Soc.  {\bf 23}  (2010),  no. 2, 405--468,

%\bibitem[BdFF]{BdFF}
%S. Boucksom, T. de Fernex, C. Favre.
%{\it Valuations and multiplier ideals}.
%Manuscript in preparation.


\bibitem[BDPP04]{BDPP}
S. Boucksom, J.-P. Demailly, M. Paun and Th. Peternell.
{\it The pseudo-effective cone of a compact K\"ahler manifold and varieties of negative Kodaira dimension}. 
Preprint (2004) {\tt math.AG/0405285}.
To appear in J. Alg. Geom. 

\bibitem[BFJ08]{BFJ}
S. Boucksom, C. Favre and M. Jonsson,
{\it Valuations and plurisubharmonic singularities}.
Publ. Res. Inst. Math. Sci. {\bf 44}  (2008),  no. 2, 449--494.

\bibitem[BFJ11]{BFJ11}
S. Boucksom, C. Favre and M. Jonsson.
{\it Non-archimedean pluripotential theory over a smooth point}.
Manuscript in preparation. 

\bibitem[Cor07]{Corti}
Alessio Corti (Ed.)
{\it Flips for 3-folds and 4-folds}. 
Oxford Lecture Series in Mathematics and its Applications {\bf 35}. 
Oxford University Press, Oxford, 2007

\bibitem[CLS11]{CLS}
D. Cox, J. Little and H. Schenck. 
{\it Toric varieties.} 
Graduate Studies in Mathematics, 124. American Mathematical Society, Providence, RI, 2011. xxiv+841 pp.

\bibitem[Cut00]{Cut} 
S.D. Cutkosky. 
{\it Irrational asymptotic behaviour of Castelnuovo-Mumford regularity}.  
J. Reine Angew. Math.  {\bf 522}  (2000), 93--103.

\bibitem[dFH09]{dFH}
T. de Fernex and C. Hacon.
{\it Singularities on normal varieties}.
Compos. Math. {\bf 145}  (2009),  no. 2, 393--414.

\bibitem[DEL00]{DEL}
J.-P. Demailly, L. Ein and R. Lazarsfeld,
{\it A subadditivity property of multiplier ideals}.
Michigan Math. J. {\bf 48}  (2000), 137--156.

\bibitem[EGA4]{EGA4}
A. Grothendieck, A. 
{\it \'El\'ements de g\'eom\'etrie alg\'ebrique IV: Etude locale des sch\'emas et des morphismes de sch\'emas.}
Inst. Hautes \'Etudes Sci. Publ. Math. {\bf 32}, 1967.

\bibitem[ELS03]{ELS}
L. Ein, R. Lazarsfeld, and K. E. Smith. 
{\it Uniform approximation of Abhyankar valuation ideals in smooth function fields.}
Amer. J. Math.  {\bf 125}  (2003),  no. 2, 409--440.

\bibitem[Eis10]{eisenstein}
E. Eisenstein.
{\it Generalization of the restriction theorem for multiplier ideals.}
Preprint (2010) {\tt arXiv:1001.2841}

\bibitem[Fakh03]{Fakh} N. Fakhruddin.
{\it Questions on self maps of algebraic varieties.} 
J. Ramanujan Math. Soc.  {\bf 18} (2003), 109--122.

\bibitem[Fav10]{Fav}
C. Favre.
{\it Holomorphic endomorphisms of singular rational surfaces.}
 Publ. Mat.  {\bf 54}  (2010),  no. 2, 389--432

\bibitem[FN05]{fuji-naka}
Y. Fujimoto, N. Nakayama.
{\it Compact complex surfaces admitting non-trivial surjective endomorphisms.}
Tohoku Math. J. (2) {\bf 57} (2005), no. 3, 395--426. 

\bibitem[Fulg11]{Fulg}
M. Fulger.
{\it Local volumes on normal algebraic varieties.}
Preprint (2011) {\tt arXiv:1105.2981}.

\bibitem[Fult93]{fulton}
W. Fulton.
{\it Introduction to toric varieties.}
Annals of Mathematics Studies {\bf 131}. 
Princeton University Press, Princeton, NJ, 1993.

%\bibitem[Gan]{ganter}
%F. M. Ganter.
%{\it $-P\cdot P$ for surfaces $z^n=f(x,y)$.} 
%Comm. Algebra {\bf 23} (1995), no. 3, 1171--1199. 

\bibitem[Gan96]{ganter}
F. M. Ganter.
{\it Properties of $-P\cdot P$ for Gorenstein surface singularities.}
Math. Z. 223 (1996), no. 3, 411–419. 

\bibitem[Gra62]{Gra} 
H. Grauert.
{\it \"Uber Modifikationen und exzeptionelle analytische Mengen}.
Math. Ann. {\bf 146} (1962), 331--368.

%\bibitem[GH]{GH}
%P. Griffiths, J. Harris.
%{\it Principles of algebraic geometry}. 
%Wiley Classics Library. John Wiley and Sons, Inc., New York, 1994. 

%\bibitem[How01]{howald}
%J. A. Howald.
%{\it Multiplier ideals of monomial ideals.}
%Trans. Amer. Math. Soc.  {\bf 353}  (2001),  no. 7, 2665--2671

%\bibitem[Ish87]{Ish}
%S. Ishii.
%{\it Isolated $\Q$-Gorenstein singularities of dimension three}, 
%in {\it Complex analytic singularities},  671--685,  Adv. Stud. Pure Math. {\bf 8}, North-Holland, %Amsterdam, 1987. 

\bibitem[Ish90]{Ish2}
S. Ishii. 
{\it The asymptotic behaviour of plurigenera for a normal isolated singularity}, 
Math. Ann. {\bf 286} (1990), 803-812. 

\bibitem[Isk03]{Isk}
V.A. Iskovskikh.
{\it $b$-divisors and Shokurov functional algebras.} 
(Russian) Tr. Mat. Inst. Steklova  {\bf 240}  (2003),  Biratsion. Geom. Linein. Sist. Konechno Porozhdennye Algebry, 8--20;  translation in  Proc. Steklov Inst. Math.  (2003),  no. 1 (240), 4-15

\bibitem[Izu81]{izumi}
S.~Izumi.
{\it Linear complementary inequalities for orders of germs of analytic functions}.
Invent. Math. {\bf 65} (1981/82), no. 3, 459--471.

\bibitem[Kaw08]{Kaw}
M. Kawakita.
{\it On a comparison of minimal log discrepancies in terms of motivic integration.} 
J. Reine Angew. Math. 620 (2008), 55--65,

\bibitem[Kn\"o73]{Kno} 
F.W. Kn\"oller.
{\it 2-dimensionale Singularit\"aten und Differentialformen.} 
Math. Ann. {\bf 206}, 205-213 (1973)

\bibitem[Kol97]{Kol}
J. Koll\'ar.
{\it Singularities of pairs}, in {\it Algebraic geometry---Santa Cruz 1995,}  221--287,
Proc. Sympos. Pure Math., {\bf 62}, Part 1, Amer. Math. Soc., Providence, RI, 1997.

\bibitem[Kol08]{kol-ex}
J. Koll\'ar.
{\it Exercices in the birational geometry of algebraic varieties,}  
Preprint (2008) {\tt arXiv:0809.2579}. 

\bibitem[Kol11]{Kol11}
J. Koll\'ar, 
{\it Book on Moduli of Surfaces.}
Ongoing project. 

\bibitem[KoMo98]{KM}
J. Koll\'ar and S. Mori.
{\it Birational geometry of algebraic varieties.}
Cambridge Tracts in Mathematics, {\bf 134}. Cambridge University Press, Cambridge, 1998.

\bibitem[K\"ur03]{Kuro} 
A. K\"uronya. 
{\it A divisorial valuation with irrational volume}.  
J. Algebra  {\bf 262}  (2003),  no. 2, 413--423. 

\bibitem[KuMa08]{KuMa} 
A. K\"uronya, C. MacLean. 
{\it Zariski decomposition of b-divisors}. 
Preprint (2008) {\tt arXiv:0807.2809}.

\bibitem[Laz04]{Laz1}
R. Lazarsfeld.
{\it Positivity in algebraic geometry, I and II.}
Ergebnisse der Mathematik und ihrer Grenzgebiete. 3.
Folge. A Series of Modern Surveys in Mathematics, {\bf 49}.
Springer-Verlag, Berlin, 2004.

\bibitem[LM09]{LM} R. Lazarsfeld, M. Musta\c{t}\v{a}.
{\it Convex bodies associated to linear series}. 
Ann. de l'ENS {\bf 42} (2009), no.~5, 783-835.

\bibitem[Loo84]{Loo} E.J.N. Looijenga.
{\it Isolated singular points on complete intersections}. 
London Mathematical Society Lecture Note Series {\bf 77}. Cambridge University Press, Cambridge, 1984.

\bibitem[Mora87]{Mora} 
M. Morales. 
{\it Resolution of quasi-homogeneous singularities and plurigenera}. 
Compos. Math. {\bf 64}, (1987) 311-327. 


%\bibitem[Morr]{morris}
%D. W. Morris.
%{\it Introduction to arithmetic groups.}
%Book. 
%\texttt{people.uleth.ca/$\sim$dave.morris}


\bibitem[Nag59]{Na}
M. Nagata,
{\it On the 14-th problem of Hilbert.}
{Amer. J. Math.} \textbf{81} (1959), 766--772.


\bibitem[Nak08]{naka}
N. Nakayama.
{\it On complex normal projective surfaces admitting non-isomorphic surjective 
endomorphisms.}
Preprint.


\bibitem[Oda88]{oda}
T. Oda.
{\it Convex bodies and algebraic geometry. An introduction to the theory of toric varieties.}  
Ergebnisse der Mathematik und ihrer Grenzgebiete (3), {\bf 15}. 
Springer-Verlag, Berlin, 1988.

\bibitem[Ram73]{Ram}
C.-P. Ramanujam.
{\it On a geometric interpretation of multiplicity.}
Invent. Math. {\bf 22} (1973/74), 63--67.

%\bibitem[Sak77]{Sak} F. Sakai.
%{\it Kodaira dimensions of complements of divisors},
%in {\it Complex analysis and algebraic
%geometry: A collection of papers dedicated to K. Kodaira}. 239-257. Iwanami-Cambridge, 1977. 

\bibitem[Sak84]{Sak84} F. Sakai.
{\it Anticanonical models of rational surfaces}. 
Math. Ann. {\bf 269} (1984), 389--410. 

%\bibitem[San86]{san}
%G.K. Sankaran.
%{\it Tsuchihashi's cusp singularities
%and automorphisms of Hilbert modular variety cusps.}
%Math. Ann. {\bf 274} (1986), no. 4, 691--698.

\bibitem[Sch09]{Sch}
K. Schwede.
{\it Test ideals in non-$\Q$-Gorenstein rings.} 
Trans. Amer. Math. Soc. 363 (2011), no. 11, 5925--5941.

\bibitem[Tak06]{Tak}
S. Takagi,
{\it Formulas for multiplier ideals on singular varieties.}
Amer. J. Math. {\bf 128}  (2006),  no. 6, 1345--1362.

%\bibitem[Tak08]{Tak2}
%S. Takagi,
%{A characteristic p analogue of plt singularities and adjoint ideals.}
%Math. Z. {\bf 259} (2008),  321--341.

\bibitem[Tak11]{Tak3}
S. Takagi.
{\it A subaddivity formula for multiplier ideals associated to log pairs.}
Preprint (2011) {\tt arXiv:1103.1179}.

\bibitem[Tsu83]{tsu}
H. Tsuchihashi.
{\it Higher-dimensional analogues of periodic continued
fractions and cusp singularities.}
Tohoku Math. J. (2)  {\bf 35}  (1983),  no. 4, 607--639.

\bibitem[Urb10]{Urb}
S. Urbinati.
{\it Discrepancies of non-$\Q$-Gorenstein varieties. }
Preprint (2010) {\tt arXiv:1001.2930}.

\bibitem[Wah90]{Wahl}
J. Wahl.
{\it A characteristic number for links of surface singularities}.
J. Amer. Math. Soc. {\bf 3} (1990), no. 3, 625--637.

\bibitem[Wat80]{Wat1}
K. Watanabe. 
{\it On plurigenera of normal isolated singularities. I.} 
Math. Ann. {\bf 250} (1980), no. 1, 65--94. 

\bibitem[Wat87]{Wat2}
K. Watanabe. 
{\it On plurigenera of normal isolated singularities. II.} 
in {\it Complex analytic singularities},  671--685,  Adv. Stud. Pure Math., 8, North-Holland, Amsterdam, 1987. 

\bibitem[Xu94]{Xu}
G. Xu,
{\it Curves in $\P^2$ and symplectic packings.}
{Math. Ann.} {\bf 299} (1994), 609--613.



\bibitem[ZS75]{ZS}
O. Zariski and P. Samuel.
{\it Commutative algebra. Vol. II.}
Reprint of the 1960 edition. Graduate Texts in Mathematics, Vol. 29.
Springer-Verlag, New York-Heidelberg, 1975.

\bibitem[dqZ06]{dqz}
D.-Q. Zhang.
{\it Polarized endomorphisms of uniruled varieties.}
Compos. Math. 146 (2010), no. 1, 145–168.

\bibitem[swZ06]{zhang}
S.-W. Zhang.
{\it Distribution in algebraic dynamics.}
Surveys in differential geometry. Vol. X, 381--430,
Surv. Differ. Geom., 10, Int. Press, Somerville, MA, 2006.



\end{thebibliography}
\end{document}